\documentclass[11pt,reqno]{amsart}

\usepackage{float}

\usepackage{placeins}

\usepackage[labelfont=bf]{caption}
\DeclareCaptionLabelSeparator{none}{ }
\makeatletter
\renewcommand{\fnum@figure}{Fig. \thefigure}
\makeatother

\usepackage{etoolbox}

\makeatletter
\newcommand*\NoIndentAfterEnv[1]{%
  \AfterEndEnvironment{#1}{\par\@afterindentfalse\@afterheading}}
\makeatother

\NoIndentAfterEnv{itemize}
\NoIndentAfterEnv{enumerate}
\NoIndentAfterEnv{proof}
\NoIndentAfterEnv{figure}
\NoIndentAfterEnv{equation}
\NoIndentAfterEnv{minted}
\NoIndentAfterEnv{Satz}
\NoIndentAfterEnv{Def}
\NoIndentAfterEnv{Ex}
\NoIndentAfterEnv{Rem}
\NoIndentAfterEnv{Algo}
\NoIndentAfterEnv{Listing}
\NoIndentAfterEnv{Prop}
\NoIndentAfterEnv{PropDef}
\NoIndentAfterEnv{Lem}
\NoIndentAfterEnv{Kor}
\NoIndentAfterEnv{Ass}
\NoIndentAfterEnv{specialSatz}

\usepackage{scalerel,stackengine}
\stackMath
\newcommand\reallywidehat[1]{%
\savestack{\tmpbox}{\stretchto{%
  \scaleto{%
    \scalerel*[\widthof{\ensuremath{#1}}]{\kern-.6pt\bigwedge\kern-.6pt}%
    {\rule[-\textheight/2]{1ex}{\textheight}}
  }{\textheight}%
}{0.5ex}}%
\stackon[1pt]{#1}{\tmpbox}%
}

\usepackage[OT2,T1]{fontenc}
\usepackage[T1]{fontenc}
\usepackage[english]{babel}
\usepackage{csquotes}
\usepackage{lmodern}
\usepackage{textcomp}

\usepackage{pythonhighlight}
\usepackage{listings}

\definecolor{codegreen}{rgb}{0,0.6,0}
\definecolor{codegray}{rgb}{0.5,0.5,0.5}
\definecolor{codepurple}{rgb}{0.58,0,0.82}
\definecolor{backcolour}{rgb}{0.95,0.95,0.92}
\lstdefinestyle{mystyle}{
  backgroundcolor=\color{backcolour}, commentstyle=\color{codegreen},
  keywordstyle=\color{magenta},
  numberstyle=\tiny\color{codegray},
  stringstyle=\color{codepurple},
  basicstyle=\ttfamily\footnotesize,
  breakatwhitespace=false,         
  breaklines=true,                 
  captionpos=b,                    
  keepspaces=true,                 
  numbers=left,                    
  numbersep=5pt,                  
  showspaces=false,                
  showstringspaces=false,
  showtabs=false,                  
  tabsize=2
}
\usepackage{amssymb, amsmath, amsthm}
\usepackage{fullpage}
\usepackage{graphicx}
\usepackage{svg}
\usepackage{calc}

\usepackage{caption}
\usepackage{subcaption}

\usepackage[bottom]{footmisc}

\usepackage{hyperref}
\usepackage{url}
\usepackage{enumerate}

\usepackage{soul}

\usepackage{amsmath}
\usepackage{amssymb}
\usepackage{amsthm}
\usepackage{stmaryrd}

\usepackage{graphicx} 

\usepackage{mathtools}

\usepackage{mathabx,epsfig}

\usepackage{stmaryrd}

\usepackage{wasysym}

\usepackage{color}
\definecolor{darkred}{RGB}{170,0,0}
\definecolor{darkgreen}{RGB}{0,120,0}

\usepackage{soul}
\setuldepth{noDescent}

\usepackage{tikz}
\usetikzlibrary{cd}
\usetikzlibrary{babel}
\tikzset{
	labl/.style={anchor=south, rotate=90, inner sep=.5mm}
}

\usepackage{tcolorbox}

\usepackage{hyperref}
\hypersetup{
	colorlinks,
	citecolor=blue,
	filecolor=black,
	linkcolor=blue,
	urlcolor=blue
}

\usepackage{booktabs}

\usepackage{nowidow}

\usepackage{calligra}
\usepackage{mathrsfs}

\usepackage{enumerate}

\usepackage[retainorgcmds]{IEEEtrantools}

\definecolor{lightgray}{rgb}{0.4,0.4,0.4}

\let\le\leqslant
\let\ge\geqslant

\DeclareMathOperator{\an}{an}

\DeclareMathOperator{\characteristic}{char}

\DeclareMathOperator{\ctt}{CTT}

\DeclareMathOperator{\interior}{int}

\DeclareMathOperator{\Nm}{Nm}

\DeclareMathOperator{\ns}{ns}

\DeclareMathOperator{\ret}{ret}

\DeclareMathOperator{\slope}{slope}

\DeclareMathOperator{\Spec}{Spec}

\DeclareMathOperator{\tail}{tail}
\DeclareMathOperator{\tame}{tame}

\DeclareMathOperator{\sHom}{\mathscr{H}\text{\kern -3pt {\calligra\large om}}\,}

\mathcode`l="8000
\begingroup
\makeatletter
\lccode`\~=`\l
\DeclareMathSymbol{\lsb@l}{\mathalpha}{letters}{`l}
\lowercase{\gdef~{\ifnum\the\mathgroup=\m@ne \ell \else \lsb@l \fi}}%
\endgroup

\renewcommand{\hat}{\widehat}
\renewcommand{\phi}{\varphi}
\renewcommand{\epsilon}{\varepsilon}

\usepackage{mathtools}

\newcommand{\abs}[1]{\left\lvert #1 \right\rvert}

\newcommand{\rest}[2]{\left.{#1}\right|_{#2}}

\newcommand{\BA}{{\mathbb{A}}}

\newcommand{\BK}{{\mathbb{K}}}

\newcommand{\BP}{{\mathbb{P}}}
\newcommand{\BQ}{{\mathbb{Q}}}
\newcommand{\BR}{{\mathbb{R}}}

\newcommand{\CH}{{\mathcal{H}}}

\newcommand{\CO}{{\mathcal O}}

\newcommand{\CX}{{\mathcal X}}
\newcommand{\CY}{{\mathcal Y}}

\newcommand{\rD}{{\mathsf D}}

\newcommand{\rU}{{\mathsf U}}

\newcommand{\rX}{{\mathsf X}}
\newcommand{\rY}{{\mathsf Y}}

\newtheorem{Prin}{Principle}[section]
\newtheorem{Prop}[Prin]{Proposition}
\newtheorem{Lem}[Prin]{Lemma}
\newtheorem{Satz}[Prin]{Theorem}

\newtheorem{Kor}[Prin]{Corollary}

\newcommand{\thistheoremname}{}
\newtheorem*{genericthm}{\thistheoremname}

\theoremstyle{definition}
\newtheorem{Def}[Prin]{Definition}
\newtheorem{Ex}[Prin]{Example}
\newtheorem{Rem}[Prin]{Remark}
\newtheorem{Algo}[Prin]{Algorithm}

\definecolor{ferrarired}{rgb}{1.0, 0.11, 0.0}
\newtcolorbox{01Def}{colback=black!0,colframe=ferrarired!50}

\usepackage[utf8]{inputenc}

\begin{document}

\title{Semistable reduction of plane quartics at $p=3$}

\begin{abstract}
We explain how to compute the semistable reduction of plane quartic curves over local fields of residue characteristic $p=3$. Our approach is based on finding suitable degree-$3$ coverings of the projective line by such plane quartics and on the different function of Cohen, Temkin, and Trushin associated to the analytifications of these coverings. In particular, we give an explicit formula for computing the different function on a given interval. The resulting algorithm for computing the semistable reduction of plane quartics is implemented in SageMath, and we illustrate it by determining the semistable reduction of a particular plane quartic at $p=3$ that arises as a quotient of the non-split Cartan modular curve $X^+_\textrm{ns}(27)$.
\end{abstract}


\author{Ole Ossen}

\email{ole.ossen@gmail.com}

\maketitle

\section*{Introduction}

\noindent Let $K$ be a field that is complete with respect to a non-trivial discrete non-archimedean valuation
\begin{equation*}
v_K\colon K\to\BR\cup\{\infty\}.
\end{equation*}
We denote the valuation ring of $K$ by $\CO_K$ and assume that the residue field $\kappa$ of $\CO_K$ is a perfect field of positive characteristic $p\coloneqq\characteristic(\kappa)$. Suppose that $Y$ is a smooth projective geometrically irreducible $K$-curve. According to the Semistable Reduction Theorem proved by Deligne and Mumford (\cite{deligne-mumford}), the base change of $Y$ to a suitable finite field extension of $K$ has semistable reduction. In this article we discuss how to compute this semistable reduction in the case that $Y$ is a plane quartic curve and $\kappa$ a field of characteristic $3$. To this end, we will explain how to find a \emph{potentially semistable model} of $Y$, meaning a model whose normalized base change to a suitable finite extension $L/K$ is a semistable model.

Our approach is based on choosing a suitable finite covering $\phi\colon Y\to\BP^1_K$. If one can find such a covering of degree smaller than the residue characteristic $p$, one may obtain a potentially semistable model of $Y$ as the normalization in $K(Y)$ of a model of $\BP_K^1$ \emph{separating the branch points} of $\phi$. This has been used in many works concerning the semistable reduction of curves, for example \cite{bouw-wewers} (which concerns the case of superelliptic curves) and \cite{ddmm} and related works (treating the case of hyperelliptic curves in great detail).

The strategy of separating branch points is not in general sufficient for computing the semistable reduction of plane quartic curves at residue characteristic $p=3$. Indeed, plane quartics admit a morphism $\phi\colon Y\to\BP_K^1$ of degree $3$, but not of degree smaller than $3$. An additional complication is that one may not in general find a degree-$3$ morphism $\phi\colon Y\to\BP_K^1$ that is \emph{galois}. While some work has been done on computing the semistable reduction of coverings of degree $p=\characteristic(\kappa)$, this has mostly been restricted to cyclic coverings of $\BP_K^1$, that is, to the case of superelliptic curves.

Coleman studied the semistable reduction of cyclic coverings $Y\to\BP_K^1$ of degree $p$ (\cite[\mbox{Section 6}]{coleman}) and introduced the notion of \emph{$p$-approximations}, used for finding a suitable equation for $Y$. Building on this, Lehr and Matignon (\cite{lehr}, \cite{matignon}, \cite{lehr-matignon}) explained how to compute the semistable reduction of $Y$ in the case of $p$-cyclic $\phi\colon Y\to\BP_K^1$ conditional on the assumption that the branch locus of $\phi$ be \emph{equidistant}. A related approach not relying on the assumption of equidistant branch locus was developed in \cite{arzdorf} and \cite{arzdorf-wewers}, a version of which is implemented in the SageMath (\cite{sagemath}) package MCLF (\cite{mclf}). 

Regarding the semistable reduction of plane quartics, mention should be made of two recent works: In \cite{super-winky-cat}, the stable reduction of plane quartics is described in terms of their \emph{Dixmier--Ohno invariants}, but conditionally on the assumption $\characteristic(\kappa)>7$. In \cite{cayley-octads}, the stable reduction of plane quartics is characterized in terms of combinatorial objects called \emph{Cayley Octads}. This approach is implemented in Magma, but as of now conjectural.

In the author's PhD thesis \cite{diss}, a general method for computing the semistable reduction of a covering $Y\to\BP_K^1$ of degree $p=\characteristic(\kappa)$ is developed. For the case of plane quartics, it is made completely explicit and, building on the MCLF package, is implemented in SageMath. The purpose of the present article is to extract the treatment of plane quartics and present it in a self-contained manner. The most important theoretical results of the first three chapters of  \cite{diss} will also appear separately (\cite{ossenwewers}).

\subsection*{Outline}

We begin by reviewing the notion of (potentially) semistable models of a $K$-curve $X$ and explain their connection to the so-called Type II valuations on the function field $K(X)$. This leads naturally to the \emph{analytification} $X^{\an}$ of $X$ in the sense of Berkovich. In \mbox{Section \ref{sec-delta}}, we collect what we need regarding the \emph{different function} of Cohen, Temkin, and Trushin (\cite{ctt}). It is a function on the analytification $Y^{\an}$ associated to a covering $\phi\colon Y\to X$ and controls the semistable reduction of $\phi$ to a large degree.

In \mbox{Section \ref{sec-quartics}}, we discuss the normal form we use for plane quartics. Following a more general recipe developed in \cite[\mbox{Chapter 3}]{diss}, we explain how to modify the equation of a given plane quartic so that it becomes suitable for reduction.
Next, we introduce the \emph{tame locus} associated to a plane quartic $Y$ with associated covering $\phi\colon Y\to\BP_K^1$. It is a subdomain of $(\BP^1_K)^{\an}$ defined in terms of the different function and contains essentially the same information regarding the semistable reduction of $\phi$. We separate the tame locus into two parts, the \emph{tail locus} and the \emph{interior locus}, and give an algorithm for computing both. For computing the interior locus, we derive a completely explicit formula for the different function in terms of valuations of certain coefficients of a suitable equation for $Y$ (Remark \ref{rem-final-delta-formula}). The boundary points of the tame locus yield a potentially semistable model of $Y$.

In the final section, we apply our results to the example of a quotient of the modular curve $X_{\ns}^+(27)$.

\subsection*{Conventions and notations}

Throughout, $K$ denotes a field that is complete with respect to a non-trivial discrete non-archimedean valuation $v_K$ of rank $1$. The ring of integers of $K$ is denoted by $\CO_K$. We assume that its residue field, denoted $\kappa$, is a perfect field of positive characteristic $p\coloneqq\characteristic(\kappa)$, but make no assumption regarding the characteristic of $K$.

We choose an algebraic closure of $K$ and denote its completion by $\BK$. There is a canonical extension of $v_K$ to $\BK$, which we also denote by $v_K$. We denote by $\pi\in K$ a choice of uniformizer for $v_K$. We fix a family of elements $\pi^s\in\BK$, where $s\in\BQ$, such that $v_K(\pi^s)=s$ and $\pi^s\pi^t=\pi^{s+t}$ for $s,t\in\BQ$.

All $K$-curves are assumed to be smooth, projective, and geometrically irreducible. We denote the function field of such a $K$-curve $X$ by $K(X)$. If $L/K$ is a field extension, we denote the function field of $X_L$ by $L(X)$. A \emph{covering} of $K$-curves $Y\to X$ is a finite separable morphism.

We will use normal italic letters (such as $X,Y$) to denote objects over $K$, curly letters (such as $\CX,\CY$) to denote objects over $\CO_K$, and serifless letters (such as $\rD,\rU$) to denote analytic curves in the sense of Berkovich.

\subsection*{Acknowledgements} I want to thank Stefan Wewers for his helpful comments on an earlier version of this paper, and Steffen Müller for giving me an opportunity to present this research in Groningen, three years after a question posed by him inspired it in the first place. Further thanks go to the anonymous referees for valuable suggestions for improving this paper.

\section{Models, valuations, and analytification}

\noindent Let $X$ be a smooth projective geometrically irreducible $K$-curve. An \emph{$\CO_K$-model} of $X$, or just a \emph{model} for short, is a normal flat and proper $\CO_K$-scheme $\CX$ with an isomorphism $\CX\otimes_{\CO_K}K\cong X$. 

A model $\CX$ is called \emph{semistable} if the special fiber $\CX_s=\CX\otimes_{\CO_K}\kappa$ is a \emph{semistable curve}. The latter means that $\CX_s$ is reduced and has only ordinary double points as singularities. We say that $X$ \emph{has semistable reduction} if a semistable $\CO_K$-model of $X$ exists. Analogous definitions hold for $\BK$-curves (where $\BK$ is the completion of an algebraic closure of $K$).

Given a covering $Y\to X$, the normalization of a model $\CX$ of $X$ in the function field $K(Y)$ is a model of $Y$. Similarly, given a finite field extension $L/K$, we define the \emph{normalized base change} of $\CX$ to be the normalization of $\CX\otimes_{\CO_K}\CO_L$ and denote it by $\CX_L$. It is an $\CO_L$-model of $X_L$.

The Semistable Reduction Theorem of Deligne--Mumford for algebraic curves (\cite{deligne-mumford}) states that given a $K$-curve $X$, there exists a finite separable field extension $L/K$ such that $X_L$ has semistable reduction.

\begin{Def}
We call an $\CO_K$-model $\CX$ of $X$ \emph{potentially semistable} if the normalized base change $\CX_L$ is a semistable model of $X_L$ for some finite extension $L/K$.  
\end{Def}

\begin{Rem}\label{rem-potentially-semistable}
Let $\CX$ be an $\CO_K$-model of $X$. By a theorem of Epp (\cite{epp}), there exists a finite extension $L/K$ such that the normalized base change $\CX_L$ has reduced special fiber. Because of this, any further base change, say to an extension $M/L$, is already normal, so equals the normalized base change (\cite[Remark 1.11]{diss}). In particular, the special fiber of $\CX_L$ is a semistable curve if and only if the special fiber of $\CX_M$ is.

It follows that the normalized base change $\CX_\BK$ is an $\CO_\BK$-model, which is semistable if and only if $\CX$ is potentially semistable.
\end{Rem}

A valuation $v$ on the function field $K(X)$ of a $K$-curve $X$ is called a \emph{Type II valuation} if it extends $v_K$ and has residue field of transcendence degree $1$ over $\kappa$. Given an $\CO_K$-model $\CX$ of $X$ and an irreducible component $Z$ of the special fiber $\CX_s$ of $\CX$, the local ring $\CO_{\CX,\xi_Z}$ at the generic point $\xi_Z$ of $Z$ is a discrete valuation ring. We denote its valuation by $v_Z$; it is a Type II valuation.

\begin{Prop}
\label{julian-prop}
\begin{enumerate}[(a)]
\item The map
\begin{equation*}
\CX\mapsto V(\CX)\coloneqq\{v_Z\mid Z\subseteq\CX_s\text{ irreducible component}\}
\end{equation*}
induces a bijection between isomorphism classes of models of $X$ and finite non-empty sets of Type II valuations on $K(X)$.
\item Let $Y\to X$ be a covering of $K$-curves. Let $\CX$ be a model of $X$ and let $\CY$ be the normalization of $\CX$ in the function field $K(Y)$. Then the valuations in $V(\CY)$ are the extensions to $K(Y)$ of the valuations in $V(\CX)$.
\end{enumerate}
\end{Prop}
\begin{proof}
See \cite[\mbox{Chapter 3} and \mbox{Section 5.1.2}]{rueth}.
\end{proof}

We will say that a model $\CX$ \emph{supports} a Type II valuation $v$ if $v$ is contained in $V(\CX)$.

\subsection*{Analytification}

We now recall some facts about \emph{analytic curves} in the sense of Berkovich (\cite[\mbox{Chapter 4}]{berkovich}). These are certain locally ringed spaces, though in this article we will only explicitly use their structure of topological spaces. For each algebraic $K$-curve $X$ there is an associated analytic curve $X^{\an}$, the \emph{analytification} of $X$ (\cite[\mbox{Section 3.4}]{berkovich}). The underlying set of $X^{\an}$ is obtained by adding to the closed points of $X$ one point for each valuation $v\colon K(X)\to\BR\cup\{\infty\}$ extending $v_K$. If the point $\xi\in X^{\an}$ corresponds to a valuation on $K(X)$, we will denote this valuation by $v_\xi$. If $v_\xi$ is a \mbox{Type II} valuation, then $\xi$ is called a \emph{Type II point}. 

Let $\xi\in X^{\an}$ be a point corresponding to the valuation $v_\xi$ on $K(X)$. The \emph{completed residue field at $\xi$}, denoted $\CH(\xi)$, is the completion of $K(X)$ with respect to $v_\xi$. The residue field of $\CH(\xi)$ is denoted $\kappa(\xi)$. If $\xi$ is a Type II point, then $\kappa(\xi)$ is of transcendence degree $1$ over $\kappa$, the residue field of $K$. In other words, it is an algebraic function field over $\kappa$. The \emph{genus} of $\xi$, denoted $g(\xi)$, is the genus of the corresponding smooth projective $\kappa$-curve, the so-called \emph{reduction curve} at $\xi$.

It will be convenient to associate a function on $K(X)$ similar to the valuations $v_\xi$ to each closed point $P\in X$:
\begin{equation*}
v_P\colon K(X)\to\BR\cup\{\pm\infty\},\qquad f\mapsto\begin{cases}
v_K(f(P))&\textrm{if $f\in\CO_{X,P}$},\\
-\infty&\textrm{otherwise}.
\end{cases}
\end{equation*}
The $v_P$ are so-called \emph{pseudovaluations}, and satisfy the usual rules for valuations for all $f\in\CO_{X,P}$.

\begin{Def}\label{def-valuative-functions}
Let $f\in K(X)^\times$ be a non-zero rational function. The \emph{valuative function} associated to $f$ is the evaluation function on $X^{\an}$
\begin{equation*}
\hat{f}\colon X^{\an}\to\BR\cup\{\pm\infty\},\qquad\xi\mapsto v_\xi(f).
\end{equation*}
\end{Def}

Type II valuations on the rational function field $K(x)$ can be described using \emph{discoids}. Given a monic irreducible polynomial $\psi\in K[x]$ and a rational number $r\in\BQ$, the set 
\begin{equation*}
\rD[\psi,r]\coloneqq\{\xi\in(\BP_K^1)^{\an}\mid v_\xi(\psi)\ge r\}
\end{equation*}
is called the \emph{discoid} with center $\psi$ and radius $r$. In addition, we also consider as a discoid
\begin{equation*}
\rD[\infty,r]\coloneqq\{\xi\in(\BP_K^1)^{\an}\mid v_\xi(x)\le -r\},
\end{equation*}
the disk of radius $r$ centered at $\infty$. For a discoid $\rD$, there is an associated Type II valuation $v_\rD$ on $K(x)$ (cf.\ \cite[\mbox{Lemma 4.50}]{rueth}), defined by
\begin{equation*}
v_\rD(f)=\inf\{v_\xi(f)\mid\xi\in\rD\},\qquad f\in K[x].
\end{equation*}
In fact, $v_\rD$ is the valuation corresponding to the unique boundary point of $\rD$ in $(\BP_{K}^1)^{\an}$. It can be shown that all Type II valuations on $K(x)$ arise in this way (\cite[\mbox{Theorem 4.56}]{rueth}).

\subsection*{Skeletons}\sloppy The projective line $(\BP_K^1)^{\an}$ is a \emph{simply connected special quasipolyhedron} (\cite[\mbox{Theorem 4.3.2}]{berkovich}). In particular, this means that $(\BP_K^1)^{\an}$ is uniquely path-connected. Thus it makes sense to consider the tree spanned by a finite set $S$ of points in $(\BP_K^1)^{\an}$; it consists of all paths connecting elements of $S$. Denoting by $\Gamma$ the tree spanned by $S$, there is a \emph{retraction map}
\begin{equation*}
\ret_\Gamma\colon(\BP_K^1)^{\an}\to\Gamma
\end{equation*}
that sends a point $\xi\in(\BP_K^1)^{\an}$ to the point $\xi'$, where $[\xi,\xi']$ is the unique path with $[\xi,\xi']\cap\Gamma=\{\xi'\}$.

Given a $K$-curve $X$, it is often convenient to work with the analytification $X^{\an}_\BK$ of the base change of $X$ to $\BK$. There is a natural morphism $X^{\an}_\BK\to X^{\an}$. If the point $\xi\in X^{\an}$ corresponds to the valuation $v_\xi$, then the points in $X^{\an}_\BK$ above $\xi$ correspond to extensions of $v_\xi$ to valuations on $\BK(X)$ extending the given valuation on $\BK$.

The combinatorial structure of analytic $\BK$-curves is captured by \emph{skeletons}. A skeleton of $X_\BK$ is a subspace $\Gamma\subset X_\BK^{\an}$ with the following properties:
\begin{enumerate}[(a)]
\item $\Gamma$ is homeomorphic to the topological realization of a finite graph
\item The vertices of this graph correspond to closed points and Type II points, among them all \mbox{Type II} points of positive genus
\item The complement $X^{\an}_\BK\setminus\Gamma$ is a disjoint union of open disks
\end{enumerate}
This definition follows \cite[\mbox{Section 3.5.1}]{ctt}. In \cite[\mbox{Definition 3.3}]{bpr}, skeletons are defined slightly differently, in terms of \emph{semistable vertex sets}, which are finite sets of Type II points whose complement consists of open disks and finitely many open anuli. The skeleton associated to a semistable vertex set is then defined to be the union of the vertex set and the skeletons of the finitely many open anuli. Such a skeleton associated to a semistable vertex set is a skeleton as defined above (\cite[\mbox{Lemma 3.4}]{bpr}). Conversely, the set of Type II vertices of a skeleton as defined above is a semistable vertex set (\cite[\mbox{Remark 3.5.2(ii)}]{ctt}).

\begin{Rem}\label{rem-skeletons-and-models}
\begin{enumerate}[(a)]
\item Let $\CX$ be an $\CO_K$-model of a $K$-curve $X$ and let $W$ be the set of extensions of elements of $V(\CX)$ to valuations on $\BK(X)$ extending the valuation on $\BK$. \emph{If $\{\xi\in X_\BK^{\an}\mid v_\xi\in W\}$ is the vertex set of a skeleton, then $\CX$ is potentially semistable.} To see this, note that by \cite[\mbox{Theorem 4.11}]{bpr} and \cite[\mbox{Remark 4.2(2)}]{bpr} we have $W=V(\CX')$ for a semistable model $\CX'$ of $X_\BK$. Now \cite[Theorem 2]{green} (a version of Proposition \ref{julian-prop} that works over $\BK$ instead of $K$) shows that $\CX'$ is the normalized base change $\CX_\BK$. By Remark \ref{rem-potentially-semistable}, $\CX$ is potentially semistable.

\item In particular, we see that a model $\CX$ of $\BP^1_K$ is potentially semistable if the inverse image of $\{\xi\in(\BP^1_K)^{\an}\mid v_\xi\in V(\CX)\}$ in $(\BP_\BK^1)^{\an}$ is the vertex set of a tree. Note that for this to hold it is not sufficient that $\{\xi\in(\BP^1_K)^{\an}\mid v_\xi\in V(\CX)\}$ be the vertex set of a tree, since elements of this set may split into several points in $(\BP^1_\BK)^{\an}$ (see \cite[Remark 5.8]{diss} for a concrete example). However, this need not concern us much, since the MCLF package (\cite{mclf}) has an inbuilt function \texttt{permanent\_completion} that refines a given model of $\BP^1_K$ to a potentially semistable model. We refer to the documentation of said function for more information.
\end{enumerate}
\end{Rem}

Next, let us suppose that we are given a covering $\phi\colon Y\to\BP_K^1$. We say that a model $\CX$ of $\BP_K^1$ is \emph{$\phi$-semistable} if the normalization of $\CX$ in the function field $K(Y)$ is a semistable model. We say that $\CX$ is \emph{potentially $\phi$-semistable} if this normalization is potentially semistable.

Given an $\CO_K$-model $\CX$ of $\BP^1_K$, there is a \emph{specialization map} associating to each closed point on $\BP^1_K$ a closed point on the special fiber $\CX_s$. It may be defined as follows: By the Valuative Criterion of Properness, a closed point $\Spec L\to\CX$ (where $L/K$ is a finite extension) extends to a morphism $\Spec\CO_L\to\CX$; the specialization is the image of the closed point in $\CX$ of the closed point of $\Spec\CO_L$.

\begin{Def}
Let $S\subset\BP_K^1$ be a finite set of closed points. A model $\CX$ of $\BP_K^1$ is said to \emph{separate the set $S$} if for every finite extension $L/K$, the points contained in $S_L\subset\BP_L^1$ specialize to pairwise distinct points on $(\CX_L)_s$.
\end{Def}

Given any finite set $S$ of closed points, a model separating $S$ exists. For this, one takes the tree in $(\BP_\BK^1)^{\an}$ spanned by the set $S_\BK$ and an arbitrary Type II point. It is a skeleton of $\BP_\BK^1$. By taking the images in $\BP_K^1$ of all Type II vertices of this skeleton, we obtain via \mbox{Proposition \ref{julian-prop}} a model of $\BP_K^1$ that separates the set $S$. We refer to \cite[\mbox{Sections 3--4}]{bouw-wewers} for more background on describing a model separating a given set of closed points.

\begin{Rem}\label{rem-admissible-not-enough}
If $p=\characteristic(\kappa)>\deg(\phi)$, it can be shown that every model separating the branch locus of $\phi$ is potentially $\phi$-semistable. A self-contained proof of this fact may be found in \cite[\mbox{Section 3}]{helminck}. As mentioned in the introduction, it is the basis of many works on semistable reduction of curves. In the sequel, we will consider the case $p=\deg(\phi)$, where it is not in general true that models separating the branch locus of $\phi$ are potentially $\phi$-semistable. We will explain how one can proceed in the case that $p=3$ and $\phi$ is a degree-$3$ morphism from a plane quartic curve.
\end{Rem}

\section{The different function}\label{sec-delta}

\noindent In \cite{ctt}, Cohen, Temkin, and Trushin associate a \emph{different function} $\delta^{\ctt}_\phi$ to a given covering $\phi\colon\rY\to\rX$ of analytic curves over an algebraically closed complete non-archimedean field. It is a function on $\rY$ that measures the wildness of the ramification of $\phi$. In \mbox{Section \ref{sec-interior-locus}}, we will see how to explicitly compute it on certain intervals. We apply the concept of the different to a given degree-$p$ covering $\phi\colon Y\to\BP_K^1$ of algebraic $K$-curves; it is then a function
\begin{equation*}
\delta_\phi^{\ctt}\colon Y_{\BK }^{\an}\to[0,1].
\end{equation*}
In fact, since we prefer working with additive valuations instead of absolute values, we also use a rescaled logarithm of the function $\delta^{\ctt}_\phi$. Denoting by $\delta_\phi\colon Y_{\BK }^{\an}\to[0,\infty]$ the different function we will use, we have the defining relations
\begin{equation*}
\delta_\phi^{\ctt}=\exp\big(-\frac{p-1}{p}\delta_\phi\Big),\qquad\delta_\phi=-\frac{p}{p-1}\log(\delta_\phi^{\ctt}).
\end{equation*}
Since $\phi$ is assumed to be of degree $p$, for every point $\xi\in(\BP^1_{\BK})^{\an}$ one of the following is true (\cite[\mbox{Section 4.1.1}]{ctt}):
\begin{enumerate}[(a)]
\item $\delta_\phi(\eta)=0$ for every point $\eta\in\phi^{-1}(\xi)$
\item The fiber $\phi^{-1}(\xi)$ contains only a single point
\end{enumerate}
Thus it makes sense to regard $\delta_\phi$ not only as a function on $Y^{\an}_{\BK }$, but also on $(\BP^1_{\BK})^{\an}$, by simply declaring that $\delta_\phi(\xi)\coloneqq\delta_\phi(\eta)$, whenever $\eta\in\phi^{-1}(\xi)$. Finally, consider the base change morphism
\begin{equation*}
\pi\colon Y_\BK^{\an}\to Y^{\an}.
\end{equation*}
The different function is constant on fibers of $\pi$. Indeed, given $\eta\in Y^{\an}$, the absolute Galois group of $K$ acts transitively on elements of the fiber $\eta'\in\pi^{-1}(\eta)$ (\cite[Corollary 1.3.6]{berkovich}). Since $\delta_\phi(\eta')$ depends only on the extension of valued fields $\CH(\eta')/\CH(\phi(\eta'))$ (cf.\ \mbox{Section \ref{sec-interior-locus}} below), we get the desired constancy.

Thus we will also consider $\delta_\phi$ as a function on $Y^{\an}$, by defining $\delta_\phi(\eta)$ to be the constant value of $\delta_\phi$ on $\pi^{-1}(\eta)$. Similarly we regard $\delta_\phi$ as a function on $(\BP_K^1)^{\an}$ as well.

A \emph{branch} at a point $\xi$ of an analytic curve is an equivalence class of germs of intervals starting at $\xi$. The function $\delta_\phi^{\ctt}$ is \emph{piecewise monomial} (\cite[\mbox{Corollary 4.1.8}]{ctt}), hence the function $\delta_\phi$ we use is \emph{piecewise affine}. See \cite[\mbox{Section 2.1}]{diss} and in particular \cite[\mbox{Remark 2.5}]{diss} for a more detailed discussion of this. For the present article, it is only important to know that we may associate a \emph{slope} $\slope_v(\delta_\phi)$ to any branch $v$ at a Type II point $\eta\in Y^{\an}_\BK$. Concretely, if $v$ is represented by the interval $[\eta,\eta']$, we choose a \emph{radius parametrization}
\begin{equation*}
[\eta,\eta']\to[a,b]\subset\BR
\end{equation*}
(\cite[\mbox{Section 3.6.2}]{ctt}, \cite[\mbox{Definition 1.40}]{diss}). Under this parametrization, $\delta_\phi$ corresponds to a function of the form $t\mapsto\alpha t+\beta$; here $\alpha\in\BQ$ is the slope $\slope_v(\delta_\phi)$.

\subsection*{Topological ramification} Let $\xi\in(\BP^1_{\BK})^{\an}$ and $\eta\in Y^{\an}_{\BK }$ be points with $\phi(\eta)=\xi$.

\begin{Def}\label{def-topological-ramification-index}
The \emph{topological ramification index} of $\phi$ at $\eta$ is the degree
\begin{equation*}
n_\eta\coloneqq[\CH(\eta):\CH(\xi)].
\end{equation*}
\end{Def}

\begin{Rem}\label{rem-geometric-ramification-index-type-ii}
If $\xi$ and $\eta$ are of Type II, then we have $n_\eta=[\kappa(\eta):\kappa(\xi)]$. Indeed, by \cite[\mbox{Theorem 6.3.1(iii)}]{temkin}, $\CH(\xi)$ is \emph{stable}, meaning that
\begin{equation*}
n_\eta=[\kappa(\eta):\kappa(\xi)]\cdot[\Gamma_{v_\eta}:\Gamma_{v_\xi}],
\end{equation*}
where $\Gamma_{v_\eta}$ and $\Gamma_{v_\xi}$ denote the value groups of $v_\eta$ and $v_\xi$ respectively.
But since $\eta$ and $\xi$ are \mbox{Type II}, both $\Gamma_{v_\eta}$ and $\Gamma_{v_\xi}$ equal the value group $\Gamma_\BK$ of $v_K\colon\BK\to\BR\cup\{\infty\}$. Thus $[\Gamma_{v_\eta}:\Gamma_{v_\xi}]=1$.
\end{Rem}

\begin{Def}
If $n_\eta>1$, then $\eta$ is called a \emph{topological ramification point} of $\phi$ and $\xi=\phi(\eta)$ is called a \emph{topological branch point}.

If $n_\eta$ equals $p$, then $\eta$ is called a \emph{wild topological ramification point} of $\phi$ and $\xi$ is called a \emph{wild topological branch point}. If $\eta$ is a topological ramification point, but not a wild topological ramification point, then it is called a \emph{tame topological ramification point} and $\xi$ a \emph{tame topological branch point}.
\end{Def}

\begin{Rem}\label{rem-top-ram-closed}
\begin{enumerate}[(a)]
\item The function $\eta\mapsto n_\eta$ is upper semicontinuous with at most finitely many discontinuity points on a given interval (\cite[\mbox{Lemma 3.6.10}]{ctt}). In particular, the topological ramification locus and the wild topological ramification locus associated to $\phi$ are closed. Likewise, the (wild) topological branch locus, which is the image of the (wild) topological ramification locus under the finite morphism $\phi$, is closed.
\item Part (a) implies that if $v$ is a branch at $\eta\in Y^{\an}_\BK$ represented by the path $[\eta,\eta']$, then the topological ramification index is constant on $(\eta,\eta'']$ with $\eta''$ chosen sufficiently close to $\eta$. Thus it makes sense to associate a topological ramification index to $v$, which we denote $n_v$.
\item For Type II points $\xi\in(\BP_\BK^1)^{\an}$, $\eta\in Y_\BK^{\an}$ with $\phi(\eta)=\xi$, \cite[\mbox{Lemma 4.2.2}]{ctt} shows that $\delta_\phi(\eta)>0$ if and only if the extension of completed residue fields $\CH(\eta)/\CH(\xi)$ is wildly ramified. In particular, we have $\phi^{-1}(\xi)=\{\eta\}$ if $\delta_\phi(\eta)>0$. If $\delta_\phi(\eta)=0$ and $\eta$ is a wild topological ramification point, then the extension $\CH(\eta)/\CH(\xi)$ is unramified of degree $p$.

It is instructive to interpret this in terms of special fibers of certain models of the curves $\BP^1_\BK$ and $Y_\BK$. Suppose that $v_\xi$ corresponds to a component $Z$ of a model of $\BP^1_\BK$ and $v_\eta$ corresponds to a component $W$ of a model of $Y_\BK$. If $\delta_\phi(\eta)>0$, then $W$ is the only component mapping to $Z$, and the corresponding extension of function fields is purely inseparable. In case $\eta$ is a wild topological ramification point with $\delta_\phi(\eta)=0$, we still have that $W$ is the only component above $Z$, but the induced map $W\to Z$ is now generically \'etale of degree $p$. 
\end{enumerate}
\end{Rem}

Let $x_0\colon\Spec\BK\to\BP_K^1\setminus\{\infty\}$ be a $\BK$-valued point that is not among the branch points of the induced map $Y_\BK\to\BP^1_\BK$. We may regard $x_0$ as a point on $(\BP^1_{\BK})^{\an}$. In the following we will often use the following notation: For $r\in\BQ$, we denote by $\xi_r$ the boundary point of $\rD[x_0,r]\subset(\BP^1_\BK)^{\an}$, the closed disk of radius $r$ centered at $x_0$. It corresponds to the valuation
\begin{equation*}
v_r\coloneqq v_{\xi_r}\colon\BK(x)\to\BR\cup\{\infty\},\qquad\sum_{i=0}^ka_i(x-x_0)^i\mapsto\min_{i\ge0}\big\{v_K(a_i)+ir\big\}.
\end{equation*}
The definition of $\xi_r$ depends on the $\BK$-valued point $x_0$, though the notation does not reflect this. Since we will usually keep a choice of $x_0$ fixed, this should cause no confusion.

For the following definition, we denote by $\Gamma_{0,\BK}$ the tree in $(\BP^1_\BK)^{\an}$ spanned by $\infty$ and the branch points of the map $Y_\BK\to\BP^1_\BK$. Then $x_0$ is a closed point not on $\Gamma_{0,\BK}$. The connected component of $(\BP_\BK^1)^{\an}\setminus\Gamma_{0,\BK}$ containing $x_0$ is an open disk.

\begin{Def}
We denote by $\mu(x_0)$ the radius of this disk.
\end{Def}

\begin{Def}
Let again $x_0$ be a $\BK$-valued point as above. We associate to $x_0$ another quantity $\lambda(x_0)$ depending on the given covering $\phi\colon Y\to\BP_K^1$, namely
\begin{equation*}
\lambda(x_0)=\min\big\{r\ge\mu(x_0)\mid\delta_\phi(\xi_r)=0\big\}.
\end{equation*}
Note that this minimum is well-defined since $\delta_\phi(\xi_r)=0$ for $r$ large enough by \cite[\mbox{Theorem 4.6.4}]{ctt}.
\end{Def}

By definition, we have $\lambda(x_0)\ge\mu(x_0)$, with equality if and only if $\delta_\phi(\xi_{\mu(x_0)})=0$. Another way to interpret $\lambda(x_0)$ and $\mu(x_0)$ is to note that the largest open disk centered at $x_0$ disjoint from $\Gamma_{0,\BK}$ has radius $\mu(x_0)$; and the largest open disk centered at $x_0$ disjoint from both $\Gamma_{0,\BK}$ and the wild topological branch locus of $\phi$ has radius $\lambda(x_0)$.

In \cite[\mbox{Section 2.5}]{diss}, $\lambda$ is defined slightly differently and as a function on all of $(\BP_{\BK }^1)^{\an}$. It is shown in \cite[\mbox{Proposition 2.31}]{diss} that the two definitions coincide.

\section{Plane quartic curves}\label{sec-quartics}

\noindent From now on and for the rest of this article we assume that the residue characteristic $p=\characteristic(\kappa)$ equals $3$. We begin by deriving a normal form for plane quartics over $K$ (as usual, assumed to be smooth and geometrically irreducible) with a known rational point. Let $Y$ be such a curve; by definition, it is cut out of $\BP^2_K$ by a homogeneous polynomial
\begin{equation*}
\sum_{i+j+k=4}a_{i,j,k}x^iy^jz^k,\qquad a_{i,j,k}\in K.
\end{equation*}
Without loss of generality we assume that the given rational point is the point at infinity $[0:1:0]$. Then we must have $a_{0,4,0}=0$, and $Y$ being smooth at $[0:1:0]$ implies $(a_{1,3,0}x+a_{0,3,1}z)\ne0$. We may assume that $a_{1,3,0}=0$ and $a_{0,3,1}\ne0$. Indeed, if $a_{0,3,1}\ne0$ we may replace $z$ with $z-\frac{a_{1,3,0}}{a_{0,3,1}}x$ to eliminate the $xy^3$-term. And if $a_{0,3,1}=0$, but $a_{1,3,0}\ne0$, we can simply swap $x$ and $z$.

Dehomogenizing by putting $z=1$, we see that an affine chart of $Y$
\begin{equation*}
\Spec K[x,y]/(F)\subset\BA_K^2
\end{equation*}
is cut out by a polynomial of the form
\begin{equation}\label{equ-quartic-normal-form}
F=y^3 + A_2y^2 + A_1y + A_0,
\end{equation}
where $A_2,A_1,A_0\in K[x]$ are polynomials of degree at most $2$, at most $3$, and at most $4$ respectively. In fact, in case $\characteristic(K)\ne3$, it is possible to eliminate the term $A_2y^2$, as is done for the normal form of \cite[\mbox{Section 1}]{shioda}, but we do not need this. The function field $K(Y)$ of $Y$ is a cubic extension of the rational function field $K(x)$, whose generator $y$ has minimal polynomial $F$. The corresponding morphism 
\begin{equation*}
\phi\colon Y\to\BP_K^1
\end{equation*}
is a degree-$3$ covering of the projective line. Denoting by $\Delta_F$ the discriminant of $F$, the morphism $\phi$ is \'etale over $U\coloneqq D(\Delta_F)\subseteq\BA_K^1$. We note that smoothness of $Y$ implies $\Delta_F\ne0$, and so $\phi$ is generically \'etale. We write $S\coloneqq\CO_Y(\phi^{-1}(U))$ (the localization of $K[x,y]/(F)$ at $\Delta_F$).

In the present situation, there is a simple formula for the quantity $\mu(x_0)$ associated to the covering $\phi$ and a $\BK$-valued point $x_0\colon\Spec\BK\to U$. Writing
\begin{equation*}
d_i\coloneqq\frac{\Delta_F^{(i)}}{i!},\qquad i\ge0
\end{equation*}
(where $\Delta_F^{(i)}$ denotes the $i$-th derivative of the polynomial $\Delta_F$), it is an easy consequence of the theory of the Newton polygon that
\begin{equation*}
\mu(x_0)=\max_{i\ge1}\big\{\frac{\hat{d}_0(x_0)-\hat{d}_i(x_0)}{i}\big\}.
\end{equation*}
Here $\hat{d}_i$ denotes the valuative function associated to $d_i$ (Definition \ref{def-valuative-functions}). We refer to \cite[\mbox{Lemma 2.14}]{diss} for details.

To also determine the quantity $\lambda(x_0)$ associated to the $\BK$-valued point $x_0$, it will be necessary to modify the equation for $Y$. To be precise, we will switch to a coordinate system where a given point $p_0\in\phi^{-1}(x_0)$ has the coordinates $(0,0)$; the resulting equation for $Y$ is \eqref{equ-desired-form} below. To keep track of how the coefficients of this new equation vary as a function of $p_0$, we begin by replacing the coefficients of $F$ by their Taylor expansions. These are polynomials in a new variable $t$ with coefficients in $K[x]$:
\begin{equation*}
\tilde{A}_2=\sum_{l=0}^2\frac{A_2^{(l)}}{l!}t^l,\qquad\tilde{A}_1=\sum_{l=0}^3\frac{A_1^{(l)}}{l!}t^l,\qquad\tilde{A}_0=\sum_{l=0}^4\frac{A_0^{(l)}}{l!}t^l.
\end{equation*}
Next, we introduce a second new variable $T$. With the goal of obtaining an equation for $Y$ for which the constant coefficient has terms only in degrees $2,3,4$ (cf.\ the proof of Lemma \ref{lem-black-box} below), we define
\begin{equation*}
\tilde{H}(T)\coloneqq T^3+\tilde{A}T^2+\tilde{B}T+\tilde{C}\coloneqq\tilde{F}(T+u),\qquad u=y-\frac{y^2A_2'+yA_1'+A_0'}{3y^2+2yA_2+A_1}t,
\end{equation*}
and introduce notation for the coefficients of $\tilde{A},\tilde{B},\tilde{C}$:
\begin{equation}\label{equ-need-for-valuative-functions}
\tilde{A}=\sum_{l=0}^2a_lt^l,\qquad \tilde{B}=\sum_{l=0}^3b_lt^l,\qquad \tilde{C}=\sum_{l=0}^4c_lt^l
\end{equation}
Since the denominator appearing in the definition of the function $u$ is just the derivative of $F$, it follows that $u$, and therefore also all the functions $a_l,b_l,c_l$, are regular functions on $\Spec S$.

We now fix a $\BK $-valued point $p_0\colon\Spec\BK \to\Spec S$. By evaluating the coefficients of $\tilde{A},\tilde{B},\tilde{C}$ at $p_0$, we obtain a polynomial 
\begin{equation}\label{equ-desired-form}
H(T)\coloneqq T^3+AT^2+BT+C \coloneqq T^3+\Big(\sum_{l=0}^2a_l(p_0)t^l\Big)T^2+\Big(\sum_{l=0}^3b_l(p_0)t^l\Big)T+\sum_{l=0}^4c_l(p_0)t^l,
\end{equation}
whose coefficients $A,B,C$ lie in $\BK[t]$. Since $F$ is geometrically irreducible, $H$ is irreducible as well. Let us denote by $w\in \BK(Y)$ a zero of $H$. It is a generator of the field extension $\BK (Y)/\BK (x)$.

\mbox{Lemma \ref{lem-black-box}} below summarizes what we know about the generator $w$ in case $\lambda(x_0)>\mu(x_0)$, where $x_0=\phi(p_0)$. We first need to introduce some notation. For $r\ge\mu(x_0)$, we denote as before by $\xi_r$ the boundary point of the disk $\rD[x_0,r]$ of radius $r$ centered at $x_0$. It follows from the definition of $\lambda(x_0)$ that $\delta_\phi(\xi_r)>0$ for $\mu(x_0)<r<\lambda(x_0)$. Because the wild topological ramification locus is closed, $\xi_{\lambda(x_0)}$ is a wild topological branch point; we denote its unique preimage in $Y^{\an}_{\BK }$ by $\eta$.

The assumption $\lambda(x_0)>\mu(x_0)$ is crucial: Without it, we could not infer that $\xi_{\lambda(x_0)}$ is a wild topological branch point, and ultimately this is what guarantees a reduction curve of Artin--Schreier type \eqref{equ-artin-schreier-polynomial}. We refer to \cite[\mbox{Section 3.3}]{diss} for more details.

\begin{Lem}\label{lem-black-box}
Suppose that $p_0$ is a $\BK $-valued point in $\Spec S$ whose image $x_0=\phi(p_0)$ satisfies $\lambda(x_0)>\mu(x_0)$. Then the following are true:
\begin{enumerate}[(a)]
\item \begin{equation*}
\lambda(x_0)=\max_{k\in\{2,3,4\}}\big\{\frac{3\hat{b}_0(p_0)-2\hat{c}_k(p_0)}{2k}\big\}
\end{equation*}
\item The reduction $\overline{w}\coloneqq\overline{w\pi^{-v_\eta(w)}}$ is a generator of the field extension of residue fields $\kappa(\eta)/\kappa(\xi_{\lambda(x_0)})$
\item The minimal polynomial of $\overline{w}$ is the Artin--Schreier polynomial
\begin{equation}\label{equ-artin-schreier-polynomial}
T^3+\overline{b_0(p_0)\pi^{-2s}}T+\overline{C\pi^{-3s}},
\end{equation}
where $s=v_\eta(w)=\frac{v_\eta(C)}{3}$
\item For $k\in\{2,3,4\}$, the maximum in (a) is achieved for $k$ if and only if $\overline{C\pi^{-3s}}$ has non-vanishing degree-$k$ coefficient
\end{enumerate}
\end{Lem}
\begin{proof}
The key point is that by construction, the polynomial $C$ is of degree $\le4$ with vanishing constant and degree-$1$ terms. To wit, the constant coefficient of $C$ is
\begin{equation*}
y(p_0)^3+A_2(p_0)y(p_0)^2+A_1(p_0)y(p_0)+A_0(p_0)=0.   
\end{equation*}
And writing
\begin{equation*}
v\coloneqq\frac{y(p_0)^2A_2'(p_0)+y(p_0)A_1'(p_0)+A_0'(p_0)}{3y(p_0)^2+2y(p_0)A_2(p_0)+A_1(p_0)},
\end{equation*}
the degree-$1$ coefficient of $C$ is
\begin{equation*}
y(p_0)^2A_2'(p_0)+y(p_0)A_1'(p_0)+A_0'(p_0)-v\cdot(3y(p_0)^2+2y(p_0)A_2(p_0)+A_1(p_0))=0.    
\end{equation*}
Thus $w$ satisfies Assumption 3.23 of \cite[\mbox{Section 3.3}]{diss} for $m=2$. 

Now everything follows from general results proved in the first three chapters of \cite{diss}. As mentioned in the introduction, the results of these chapters will also appear in a separate article \cite{ossenwewers}.

Part (a) follows from \cite[\mbox{Theorem 3.30}]{diss}, (b) follows from \cite[\mbox{Lemma 3.27}]{diss}, and (d) is part of \cite[\mbox{Remark 3.31}]{diss}. According to \cite[\mbox{Remark 3.28}]{diss}, the minimal polynomial of $\overline{w}$ is
\begin{equation*}
T^3+\overline{A\pi^{-s}}T^2+\overline{B\pi^{-2s}}T+\overline{C\pi^{-3s}}   
\end{equation*}
and has constant discriminant. The latter easily implies that $\overline{A\pi^{-s}}=0$ and that $\overline{B\pi^{-2s}}$ is a constant (see for example \cite[\mbox{Lemma 4.3}]{diss}).
\end{proof}

\begin{Rem}\label{rem-read-off-genus}
The genus of the smooth $\kappa$-curve defined by the equation \eqref{equ-artin-schreier-polynomial} can be read off from the polynomial 
\begin{equation*}
\overline{C\pi^{-3v_\eta(w)}}=\overline{c}_4\overline{t}^4+\overline{c}_3\overline{t}^3+\overline{c}_2\overline{t}^2
\end{equation*}
(where $\overline{t}=\overline{t\pi^{-v_\xi(t)}}$, and where $\overline{c}_4,\overline{c}_3,\overline{c}_2$ are certain constants). Indeed, it follows from \cite[\mbox{Theorem 3.7.8(d)}]{stichtenoth} that \eqref{equ-artin-schreier-polynomial} defines a curve of genus $3$ if $\overline{c}_4\ne0$; it defines a curve of genus $1$ if $\overline{c}_4=0$, but $\overline{c}_2\ne0$; and it defines a curve of genus $0$ otherwise.

\end{Rem}

\section{Tame locus}\label{sec-tame-locus}

\noindent Let $Y$ be a plane quartic curve in the normal form \eqref{equ-quartic-normal-form} discussed in the previous section and let $\phi\colon Y\to\BP_K^1$ be the associated degree-$3$ covering. As mentioned in \mbox{Remark \ref{rem-admissible-not-enough}}, a model of $\BP_K^1$ separating the branch locus of $\phi$ is not usually potentially $\phi$-semistable. The following lemma explains how to fix this. 

As before, we denote by $\Gamma_{0,\BK}$ the tree in $(\BP^1_\BK)^{\an}$ spanned by the branch locus of $Y_\BK\to\BP^1_\BK$ and the point $\infty$. We denote the image of $\Gamma_{0,\BK}$ in $(\BP^1_K)^{\an}$ by $\Gamma_0$.

\begin{Lem}\label{lem-tails}
Write $\Sigma_{0,\BK}\coloneqq\phi^{-1}(\Gamma_{0,\BK})$. Then the following statements hold:
\begin{enumerate}[(a)]
\item For every point $\eta\in Y^{\an}_\BK\setminus\Sigma_{0,\BK}$ of positive genus, there exists a unique path $[\eta,\eta_1]\subset Y_\BK^{\an}$ with $[\eta,\eta_1]\cap\Sigma_{0,\BK}=\{\eta_1\}$
\item The function $\delta_\phi$ is strictly positive on $(\eta,\eta_1]$ for each such path
\item Attaching all the paths $[\eta,\eta_1]$ to $\Sigma_{0,\BK}$ and adding to the vertex set of $\Sigma_{0,\BK}$ all points of positive genus on $\Sigma_{0,\BK}$ yields a skeleton of $Y_\BK$, which is the unique minimal skeleton containing $\Sigma_{0,\BK}$
\end{enumerate}
\end{Lem}
\begin{proof}
This follows from the ``genus formula for wide open domains'', \cite[\mbox{Theorem 6.2.7}]{ctt}, which relates the slopes of the function $\delta_\phi$ to the branch locus of $\phi$ and the points of positive genus in $Y_\BK^{\an}$. A detailed proof may be found in \cite[\mbox{Propositions 2.22} \mbox{and 2.24}]{diss}
\end{proof}

Motivated by the lemma, we denote by $\Gamma_\BK$ the tree in $(\BP_\BK^1)^{\an}$ spanned by the branch locus of $Y_\BK\to\BP_\BK^1$, by the point $\infty$, and by Type II points that are the image of a point of positive genus on $Y_\BK^{\an}$. We denote by $\Gamma$ the image of $\Gamma_\BK$ in $(\BP^1_K)^{\an}$.

\begin{Lem}\label{lem-potentially-phi-semistable}
Let $\CX$ be a potentially semistable model of $\BP_K^1$ supporting all valuations corresponding to a \mbox{Type II} vertex of $\Gamma_0$ or to a \mbox{Type II} point that is the image of a positive genus point on $Y_\BK^{\an}$. Then $\CX$ is potentially $\phi$-semistable.
\end{Lem}
\begin{proof}
This follows from \mbox{Remark \ref{rem-skeletons-and-models}(a)} and \mbox{Lemma \ref{lem-tails}(c)}.
\end{proof}

\begin{Def}
The \emph{tame locus} associated to $Y$ and $\phi$ is the subdomain
\begin{equation*}
\rU^{\tame}=\ret_\Gamma^{-1}\big(\{\xi\in\Gamma\mid\delta_\phi(\xi)=0\}\big)\subseteq(\BP^1_K)^{\an}.
\end{equation*}
\end{Def}

Suppose that $\xi\in(\BP_K^1)^{\an}$ is a leaf of $\Gamma$ that is the image of a point of positive genus on $Y_\BK^{\an}$ (which by \cite[Lemma 1.8]{diss} is necessarily contained in $\rU^{\tame}$). Then the retraction
\begin{equation*}
\ret_\Gamma^{-1}(\xi)
\end{equation*}
is a discoid, called a \emph{tail discoid}. We note that such a point $\xi$ cannot lie on $\Gamma_0$, since $\Gamma_0$ has no leaves of Type II. The union of all tail discoids, denoted $\rU^{\tail}$, is called the \emph{tail locus} associated to $Y$ and $\phi$. The complement
\begin{equation*}
\rU^{\interior}\coloneqq\rU^{\tame}\setminus\rU^{\tail}
\end{equation*}
is called the \emph{interior locus} associated to $Y$ and $\phi$.

\begin{Prop}\label{prop-potentially-phi-semistable}
Let $\CX$ be a potentially semistable model of $\BP_K^1$ supporting all valuations corresponding to a Type II vertex of $\Gamma_0$ or to a boundary point of $\rU^{\tame}$. Then $\CX$ is potentially $\phi$-semistable.
\end{Prop}
\begin{proof}
In the following proof, we again denote by $\Sigma_{0,\BK}$ the inverse image of $\Gamma_{0,\BK}$ in $Y_\BK^{\an}$. In light of \mbox{Lemma \ref{lem-potentially-phi-semistable}}, it suffices to show that every point $\xi\in(\BP_K^1)^{\an}$ that is the image of a point of positive genus on $Y_{\BK }^{\an}$ arises as a boundary point of $\rU^{\tame}$ or as a vertex of $\Sigma_{0,\BK}$.

Consider first a point $\eta$ of positive genus contained in $\Sigma_{0,\BK}$. Clearly $\eta$ is a topological ramification point of the map $Y_\BK^{\an}\to(\BP^1_\BK)^{\an}$ with $\delta_\phi(\eta)=0$. If $\eta$ is not a vertex of $\Sigma_{0,\BK}$, then there are only two branches at $\eta$ contained in $\Sigma_{0,\BK}$, which we denote $v_1,v_2$. The Riemann--Hurwitz Formula \cite[\mbox{Theorem 4.5.4}]{ctt} applied to $\eta$ reads
\begin{equation*}
2g(\eta)-2=-2n_\eta+n_{v_1}+\frac{2}{3}\slope_{v_1}(\delta_\phi)+n_{v_2}+\frac{2}{3}\slope_{v_2}(\delta_\phi)-2,
\end{equation*}
\sloppy where $n_\eta$ and $n_{v_i}$ denote the topological ramification indices at $\eta$ and $v_i$ (\mbox{Definition \ref{def-topological-ramification-index}}, \mbox{Remark \ref{rem-top-ram-closed}(b)}). Note that our slopes differ from the ones in \cite{ctt} by a factor $-2/3$, corresponding to the factor relating $\delta_\phi^{\ctt}$ and $\delta_\phi$ as explained in \mbox{Section \ref{sec-delta}}. Semicontinuity of the topological ramification index (Remark \ref{rem-top-ram-closed}(a)) shows $n_{v_1}+n_{v_2}\le 2n_\eta$, so the Riemann-Hurwitz Formula implies
\begin{equation*}
3g(\eta)\le\slope_{v_1}(\delta_\phi)+\slope_{v_2}(\delta_\phi).
\end{equation*}
Thus we must have $\slope_{v_1}(\delta_\phi)>0$ or $\slope_{v_2}(\delta_\phi)>0$, in which case the image of $\eta$ in $(\BP_K^1)^{\an}$ is a boundary point of the interior locus.

Finally consider a point $\eta$ of positive genus not contained in $\Sigma_{0,\BK}$. Again we have $\delta_\phi(\eta)=0$, while the path connecting $\eta$ and $\Sigma_{0,\BK}$ consists by \mbox{Lemma \ref{lem-tails}(b)} of wild topological ramification points. Thus the image of $\eta$ in $(\BP_K^1)^{\an}$ is a boundary point of the tail locus.
\end{proof}

To describe $\rU^{\tame}$, $\rU^{\tail}$, and $\rU^{\interior}$, we will work with $\BK $-valued points on $Y$ and $\BP_K^1$. We denote by $\rU^{\tame}_{\BK },\rU^{\tail}_{\BK },\rU^{\interior}_{\BK }$ the inverse images of $\rU^{\tame},\rU^{\tail},\rU^{\interior}$ in $(\BP^1_{\BK })^{\an}$. We have
\begin{equation*}
\rU_\BK^{\tame}=\ret^{-1}_{\Gamma_\BK}\big(\{\xi\in\Gamma_\BK\mid\delta_\phi(\xi)=0\}\big)\subseteq(\BP^1_\BK)^{\an}.
\end{equation*}
Indeed, the base change map $\pi\colon\BP^1_\BK\to\BP^1_K$ is compatible with retractions in the sense that the diagram
\begin{equation*}
\begin{tikzcd}
(\BP^1_\BK)^{\an} \arrow[r, "\pi"] \arrow[d, "\ret_{\Gamma_\BK}"'] & (\BP^1_K)^{\an} \arrow[d, "\ret_\Gamma"] \\
\Gamma_\BK \arrow[r, "\rest{\pi}{\Gamma_\BK}"']                     & \Gamma                               
\end{tikzcd}
\end{equation*}
commutes, because given a point $\xi\in(\BP^1_\BK)^{\an}$ with $\xi'\coloneqq\ret_{\Gamma_\BK}(\xi)$, the map $\pi$ sends the unique path connecting $\xi$ and $\xi'$ to the unique path connecting $\pi(\xi)$ and $\pi(\xi')$. Similarly, $\rU^{\tail}_\BK$ is the union of disks of the form $\ret_{\Gamma_\BK}^{-1}(\xi)$, where $\xi$ is a leaf of $\Gamma_\BK$ that is the image of a point of positive genus on $Y_\BK^{\an}$.

\begin{Lem}\label{lem-souped-up-lambda-distinction}
Let $x_0\colon\Spec\BK\to\BP^1_K$ be a $\BK$-valued point which is not among the vertices of $\Gamma_{0,\BK}$.
\begin{enumerate}[(a)]
\item We have $x_0\in\rU^{\interior}_{\BK }$ if and only if $\lambda(x_0)=\mu(x_0)$.
\item We have $x_0\in\rU^{\tail}_{\BK }$ if and only if $\lambda(x_0)>\mu(x_0)$ and the unique point $\eta\in Y^{\an}_{\BK }$ above $\xi_{\lambda(x_0)}$ has positive genus.
\end{enumerate}
\end{Lem}
\begin{proof}
As usual, we denote by $\xi_r$ the boundary point of the disk $\rD[x_0,r]$. Then the path connecting $x_0$ and $\Gamma_{0,\BK}$ is the set
\begin{equation*}
\{\xi_r\mid \mu(x_0)\le r\le \infty\}.
\end{equation*}
We do a case distinction based on the retraction $\ret_{\Gamma_\BK}(x_0)$.

\emph{Case 1}: $\ret_{\Gamma_\BK}(x_0)\in\Gamma_{0,\BK}$. In this case, we have $\ret_{\Gamma_\BK}(x_0)=\xi_{\mu(x_0)}$. By the definition of $\lambda(x_0)$, we have $\delta_\phi(\xi_{\mu(x_0)})=0$ if and only if $\lambda(x_0)=\mu(x_0)$, and we have $\delta_\phi(\xi_{\mu(x_0)})>0$ if and only if $\lambda(x_0)>\mu(x_0)$. In the first case, $x_0\in\rU^{\tame}_\BK$, and indeed $x_0\in\rU^{\interior}_\BK$, since $\ret_{\Gamma_\BK}(x_0)\in\Gamma_{0,\BK}$ cannot be a leaf of $\Gamma_\BK$. In the second case, $x_0\not\in\rU^{\tame}_\BK$.

\emph{Case 2}: $\ret_{\Gamma_\BK}(x_0)$ lies on the image of one of the paths $[\eta,\eta_1)$ described in \mbox{Lemma \ref{lem-tails}}, where $\eta\in Y^{\an}_\BK$ is a point of positive genus and $\eta_1$ lies above $\Gamma_0$. In this case, we have $\ret_{\Gamma_\BK}(x_0)=\xi_r$ for some $r>\mu(x_0)$. By \mbox{Lemma \ref{lem-tails}(b)}, it follows that $\lambda(x_0)>\mu(x_0)$. We have $\delta_\phi(\xi_r)=0$ if and only if $\xi_r$ is the image of the point $\eta\in Y_\BK^{\an}$ of positive genus. Thus $x_0\in\rU_\BK^{\tail}$ if $\xi_r=\phi(\eta)$ and $x_0\not\in\rU^{\tame}_\BK$ otherwise.

This concludes the proof of the lemma in every case.
\end{proof}

To compute the tame locus, we define the subdomain
\begin{equation*}
\rU=\rU_2\cup\rU_4\subseteq(\BP_K^1)^{\an},\qquad\textrm{where}
\end{equation*}
\begin{equation*}
\rU_2=\{\xi\in (\BP_K^1)^{\an}\mid 3\reallywidehat{\Nm b_0}(\xi)+4\reallywidehat{\Nm c_3}(\xi)\ge6\reallywidehat{\Nm c_2}(\xi)\} \qquad\textrm{and}
\end{equation*}
\begin{equation*}
\rU_4=\{\xi\in (\BP_K^1)^{\an}\mid 8\reallywidehat{\Nm c_3}(\xi)\ge3\reallywidehat{\Nm b_0}(\xi)+6\reallywidehat{\Nm c_4}(\xi)\},
\end{equation*}
and denote by $\rU_\BK$ the inverse image of $\rU$ in $(\BP^1_\BK)^{\an}$. (Here $\Nm$ denotes the norm map of the extension of function fields $K(Y)/K(x)$, and $b_0,c_2,c_3,c_4$ are the functions introduced in \eqref{equ-need-for-valuative-functions}, so that $\Nm b_0$ and the $\Nm c_k$ are rational functions on $\BP_K^1$.)

\begin{Satz}\label{thm-main}
We have the following inclusions:
\begin{equation*}
\rU^{\tail}\subseteq\rU\subseteq\rU^{\tame}
\end{equation*}
\end{Satz}
\begin{proof}
Since the Type I points in $(\BP_K^1)^{\an}\setminus\Gamma_0$ are a dense subset of $(\BP^1_K)^{\an}$, it suffices to show the inclusions $\rU_{\BK }^{\tail}\subseteq\rU_{\BK }\subseteq\rU^{\tame}_{\BK }$ for $\BK $-valued points $x_0$ not contained in $\Gamma_{0,\BK}$. If $x_0\in\rU^{\interior}_{\BK }$, there is nothing to show. Thus we may assume $x_0\not\in\rU^{\interior}_{\BK }$, in which case we need to show that $x_0\in\rU^{\tail}_{\BK }$ if and only if $x_0\in\rU_{\BK }$.

Choose a $\BK $-valued point $p_0$ on $Y$ with $\phi(p_0)=x_0$. As before, we denote by $\eta\in Y^{\an}_{\BK }$ the unique point above $\xi_{\lambda(x_0)}$, the boundary point of the disk of radius $\lambda(x_0)$ centered at $x_0$. By \mbox{Lemma \ref{lem-souped-up-lambda-distinction}(b)}, we have $\lambda(x_0)>\mu(x_0)$, and $x_0\in\rU_{\BK }^{\tail}$ if and only if $\eta$ has positive genus.

We may apply \mbox{Lemma \ref{lem-black-box}}, and do a case distinction based on the shape of the polynomial 
\begin{equation*}
\overline{c}_4\overline{t}^4+\overline{c}_3\overline{t}^3+\overline{c}_2\overline{t}^2=\overline{C\pi^{-v_{\lambda(x_0)}(C)}}
\end{equation*}
appearing in \mbox{Lemma \ref{lem-black-box}}. If $\overline{c}_4\ne0$, then as mentioned in \mbox{Remark \ref{rem-read-off-genus}}, the point $\eta$ is of genus $3$. Moreover, \mbox{Lemma \ref{lem-black-box}(d)} implies that
\begin{equation}\label{equ-c3-dominates}
\lambda(x_0)=\frac{3\hat{b}_0(p_0)-2\hat{c}_4(p_0)}{8}\ge\frac{3\hat{b}_0(p_0)-2\hat{c}_3(p_0)}{6}.
\end{equation}
Note that \eqref{equ-c3-dominates} also holds if we replace $p_0$ with a different point in the fiber $\phi^{-1}(x_0)$. Thus writing $\phi^{-1}(x_0)=\{p_0,p_1,p_2\}$, we have
\begin{IEEEeqnarray*}{rCcCl}
\frac{3\reallywidehat{\Nm b_0}(x_0)-2\reallywidehat{\Nm c_4}(x_0)}{8}&=&\sum_{j=1}^3\frac{3\hat{b}_0(p_j)-2\hat{c}_4(p_j)}{8}&&\\
&\ge&\sum_{j=1}^3\frac{3\hat{b}_0(p_j)-2\hat{c}_3(p_j)}{6}
&=&\frac{3\reallywidehat{\Nm b_0}(x_0)-2\reallywidehat{\Nm c_3}(x_0)}{6}.
\end{IEEEeqnarray*}
Rearranging this gives
\begin{equation}\label{equ-gen4-case}
8\reallywidehat{\Nm c_3}(x_0)\ge3\reallywidehat{\Nm b_0}(x_0)+6\reallywidehat{\Nm c_4}(x_0),
\end{equation}
so $x_0\in\rU_{\BK }$.

If $\overline{c}_4=0$, but $\overline{c}_2\ne0$, one proceeds analogously to the above to derive
\begin{equation}\label{equ-gen1-case}
3\reallywidehat{\Nm b_0}(x_0)+4\reallywidehat{\Nm c_3}(x_0)\ge6\reallywidehat{\Nm c_2}(x_0).
\end{equation}
Thus whenever $\eta$ has positive genus, we have $x_0\in\rU_{\BK }$. If on the other hand $\eta$ has genus $0$, then one finds that instead of \eqref{equ-gen4-case} and \eqref{equ-gen1-case} the opposite and strict inequalities hold. Thus $x_0\not\in\rU_{\BK }$.
\end{proof}

\section{Interior locus}\label{sec-interior-locus}

\noindent Let again $Y$ be a plane quartic in the normal form of \mbox{Section \ref{sec-quartics}}, say
\begin{equation*}
Y\colon\quad y^3+A_2y^2+A_1y+A_0=0,\qquad A_2,A_1,A_0\in K[x].
\end{equation*}
In the previous section, we have introduced the subdomain $\rU\subseteq(\BP_K^1)^{\an}$ associated to $Y$. It contains the tail locus $\rU^{\tail}$, and is contained in the tame locus $\rU^{\tame}$. Thus to determine all of $\rU^{\tame}=\rU^{\tail}\cup\rU^{\interior}$, it suffices to be able to compute $\rU^{\interior}$. We explain how to do this in this section.

Let $x_0\colon\Spec\BK \to\BP_K^1$ be a $\BK $-valued point different from $\infty$ which is fixed for the rest of this section. We will explain a method to directly compute the function $\delta_\phi$ on the path $[x_0,\infty]\subset(\BP^1_{\BK })^{\an}$.

Choose a point $p_0\in\phi^{-1}(x_0)$ and write $y_0\coloneqq y(p_0)$. As before writing $t\coloneqq x-x_0$, the generator $y-y_0$ of $\BK(Y)$ has minimal polynomial over $\BK(t)$ of the form
\begin{equation}\label{equ-unfinished-minimal-polynomial}
T^3+(a_2t^2+a_1t+a_0)T^2+(b_3t^3+b_2t^2+b_1t+b_0)T+c_4t^4+c_3t^3+c_2t^2+c_1t,
\end{equation}
where the coefficients $a_0,a_1,a_2$, $b_0,\ldots,b_3$, and $c_0,\ldots,c_4$ are certain elements of $\BK$, depending on $x_0$ and $p_0$. Finally, we consider the generator $z\coloneqq y+vt$, where $v$ is a solution to the equation
\begin{equation}\label{equ-another-cubic-equation}
v^3+a_1v^2+b_2v+c_3=0.
\end{equation}
The minimal polynomial of $z$ is then by construction of the form
\begin{equation}\label{equ-doctored-minimal-polynomial}
T^3+(a_2t^2+a_1t+a_0)T^2+(b_3t^3+b_2t^2+b_1t+b_0)T+c_4t^4+c_2t^2+c_1t.
\end{equation}
(The coefficients here are of course different from the ones in \eqref{equ-unfinished-minimal-polynomial}, we just reuse the notation for convenience. They are also different from the coefficients of \eqref{equ-desired-form}, since there we eliminated the degree-$1$ term, not the degree-$3$-term.) We write $A= a_2t^2+a_1t+a_0$, $B=b_3t^3+b_2t^2+b_1t+b_0$, $C=c_4t^4+c_2t^2+c_1t$. 

\begin{Rem}\label{rem-everything-rational}
In case $x_0$ is a closed point of $\BP^1_K$, the coefficients of $A,B,C$ lie in a finite extension of $K$, which we can explicitly describe. Indeed, a closed point $x_0$ becomes rational over a suitable finite extension of $K$; the coordinate $y_0$ lives in a further finite extension; finally we need to adjoin a root $v$ of \eqref{equ-another-cubic-equation}.
\end{Rem}

Let $\xi\in(\BP^1_{\BK })^{\an}$ be a point of Type II on the path $[x_0,\infty]$. We want to compute $\delta_\phi(\xi)$. For the moment we assume that $\delta_\phi(\xi)>0$, so that $\phi^{-1}(\xi)=\{\eta\}$ for a Type II point $\eta\in Y^{\an}_{\BK }$. The completed residue fields $\CH(\xi)$ and $\CH(\eta)$ are \emph{one-dimensional analytic fields} in the sense of \cite[\mbox{Section 2.1.1}]{ctt} over $\BK $, meaning that they are complete valued fields over $\BK $ which are finite over the completion of a subfield $\BK (x)$, $x\not\in \BK $. In \cite[\mbox{Section 2.4}]{ctt}, a \emph{different}
\begin{equation*}
\delta^{\ctt}(\CH_2/\CH_1)
\end{equation*}
is attached to every separable extension $\CH_2/\CH_1$ of one-dimensional analytic fields over $K$. In parallel with our normalization of $\delta_\phi$ in \mbox{Section \ref{sec-delta}}, we define
\begin{equation*}
\delta(\CH_2/\CH_1)\coloneqq-\frac{3}{2}\log\big(\delta^{\ctt}(\CH_2/\CH_1)\big).
\end{equation*}
By the definition of $\delta_\phi$ in \cite[\mbox{Section 4.1}]{ctt} we then have
\begin{equation*}
\delta_\phi(\xi)=\delta_\phi(\eta)=\delta\big(\CH(\eta)/\CH(\xi)\big).
\end{equation*}
We will need the following definitions (cf.\ \cite[\mbox{Section 2.1}]{ctt}).

\begin{Def}
Let $\CH$ be a one-dimensional analytic field over $K$ with valuation $v_\CH$. An element $x\in\CH\setminus \BK $ such that $\CH/\reallywidehat{\BK (x)}$ is a finite separable extension is called a \emph{parameter} for $\CH$. (Here $\reallywidehat{\BK (x)}$ denotes the completion of $\BK(x)$.)
\begin{enumerate}[(a)]
\item If $x\in\CH\setminus \BK $ is a parameter such that $\CH/\reallywidehat{\BK (x)}$ is a tame extension of valued fields, then $x$ is called a \emph{tame parameter} for $\CH$. 
\item If $x\in\CH\setminus \BK $ is a parameter such that we have
\begin{equation*}
v_\CH\big(\sum_ia_ix^i\big)=\min_i\big\{v_K(a_i)+iv_\CH(x)\big\}
\end{equation*}
for all polynomials $\sum_ia_ix^i\in\BK[x]$, then $x$ is called a \emph{monomial parameter} for $\CH$.
\end{enumerate}
\end{Def}

Our key tool for computing $\delta_\phi$ is the following:

\begin{Prop}\label{prop-formula-for-delta}
Let $\CH_2/\CH_1$ be a finite separable extension of one-dimensional analytic fields over $\BK $. Suppose that $x_1$ and $x_2$ are tame monomial parameters of $\CH_1$ and $\CH_2$ respectively. Then we have
\begin{equation*}
\frac{2}{3}\delta(\CH_2/\CH_1)=v_{\CH_2}(\tfrac{dx_1}{dx_2})+v_{\CH_2}(x_2)-v_{\CH_1}(x_1),
\end{equation*}
where $\frac{dx_1}{dx_2}$ is the unique element of $\CH_2$ for which $\frac{dx_1}{dx_2}dx_2=dx_1$.
\end{Prop}
\begin{proof}
Combine \cite[\mbox{Corollary 2.4.6(ii)}]{ctt} and \cite[\mbox{Lemma 2.1.6(ii)}]{ctt}.
\end{proof}

If $\CH=\CH(\xi)$ for a Type II point on an analytic curve, then a tame monomial parameter for $\CH$ always exists (\cite[\mbox{Theorem 2.1.3} and \mbox{Lemma 2.1.6}]{ctt}). The following lemma shows that the elements $t$ and $z$ constructed above are tame monomial parameters for every Type II point $\xi\in[x_0,\infty]$ with $\delta_\phi(\xi)>0$.

\begin{Lem}
\label{lem-tame-generators}
Let $\xi\in[x_0,\infty]$ be a Type II point with $\delta_\phi(\xi)>0$, whose unique inverse image in $Y^{\an}_\BK$ we denote by $\eta$. Then the elements $t=x-x_0\in\CH(\xi)$ and $z\in\CH(\eta)$ (constructed at the beginning of this section, with minimal polynomial \eqref{equ-doctored-minimal-polynomial}) are tame monomial parameters.
\end{Lem}
\begin{proof}
For $t$ this is clear. Indeed, $\CH(\xi)$ actually equals $\reallywidehat{\BK(t)}$, so $t$ is in particular a tame parameter. And since $v_\xi$ is a Gauss valuation centered at $x_0$, it follows that $t=x-x_0$ is a monomial parameter. Let us write
\begin{equation*}
\overline{z}\coloneqq\overline{z\pi^{-s}},\qquad s=\frac{v_\xi(C)}{3}.
\end{equation*}
Since $\phi^{-1}(\xi)=\{\eta\}$, the Newton polygon of the polynomial
\begin{equation}\label{equ-not-reduction-polynomial}
F=T^3+A\pi^{-s}T^2+B\pi^{-2s}T+C\pi^{-3s}\in\CH(\xi)[T]
\end{equation}
is a straight line. Thus it makes sense to define the reduction
\begin{equation}\label{equ-reduction-polynomial}
\overline{F}=T^3+\overline{A\pi^{-s}}T^2+\overline{B\pi^{-2s}}T+\overline{C\pi^{-3s}}\in\kappa(\xi)[T].
\end{equation}
Since $z\pi^{-s}$ is a zero of \eqref{equ-not-reduction-polynomial} it follows that $\overline{z}$ is a zero of \eqref{equ-reduction-polynomial}. Hensel's Lemma shows that $\overline{F}$ cannot have more than one distinct irreducible factor (a factorization into distinct irreducible factors would lift to a factorization of $F$). Suppose first that the irreducible factor of $\overline{F}$ is linear, say
\begin{equation*}
\overline{F}=(T-\alpha)^3=T^3-\alpha^3.
\end{equation*}
This is not possible, since $\overline{C\pi^{-3s}}$ is by construction of $C$ not a third power. Thus $\overline{F}$ must be irreducible and $\overline{z}$ is a generator of the field extension $\kappa(\eta)/\kappa(\xi)$ with minimal polynomial $\overline{F}$. Since $\delta_\phi(\xi)>0$, the extension $\kappa(\eta)/\kappa(\xi)$ is inseparable, which implies that $\overline{A\pi^{-s}}=\overline{B\pi^{-2s}}=0$.

To show that $z$ is a monomial parameter, we study the extension $v_\eta$ of $v_\xi$ to $\BK (Y)$. We use MacLane's theory of inductive valuations (as explained in \cite{maclane} or \cite[Chapter 4]{rueth}) to show that $v_\eta$ is the Gauss valuation
\begin{equation*}
\sum_iA_i(z\pi^{-s})^i\mapsto\min_i\big\{v_\xi(A_i)\big\},\qquad A_i\in\BK(t).
\end{equation*}
As explained in \cite[\mbox{Section 5}]{maclane}, it suffices to show that $F$ is a \emph{key polynomial} (\cite[\mbox{Section 2}]{maclane}) over the Gauss valuation on the polynomial ring $\BK(t)[T]$,
\begin{equation*}
\sum_iA_iT^i\mapsto\min_i\big\{v_\xi(A_i)\big\},\qquad A_i\in \BK (t).
\end{equation*}
But according to \cite[\mbox{Lemma 4.8}]{rueth}, the fact that $\overline{F}$ is irreducible implies that $F$ is a key polynomial. Thus $z$ is a monomial parameter.

To show that $z$ is a tame parameter, we show that the degree of $\CH(\eta)/\reallywidehat{\BK (z)}$ is coprime to $3$. Note that the residue field $\kappa(\eta)$ is generated over the residue field of $\BK $ by $\overline{z}$ and $\overline{t}=\overline{t\pi^{-v_\xi(t)}}$. Since
\begin{equation*}
\overline{z}^3=\overline{C\pi^{-3s}},
\end{equation*}
where $\overline{C\pi^{-3s}}=\overline{c}_4\overline{t}^4+\overline{c}_2\overline{t}^2+\overline{c}_1\overline{t}$ for certain constants $\overline{c}_4,\overline{c}_2,\overline{c}_1$, it follows that the field extension $\kappa(\eta)/\kappa(\overline{z})$ has degree dividing $4$. This degree equals the degree $\CH(\eta)/\reallywidehat{K(z)}$ (by the same argument as in \mbox{Remark \ref{rem-geometric-ramification-index-type-ii}}). Thus $z$ is a tame parameter as well.
\end{proof}

\begin{Kor}\label{cor-concrete-delta-formula}
We have
\begin{equation}\label{equ-concrete-delta-formula}
\frac{2}{3}\delta_\phi(\xi)=\min\big\{v_K(3),v_\xi(A)-\frac{v_\xi(C)}{3},v_\xi(B)-\frac{2v_\xi(C)}{3}\big\},
\end{equation}
where $A,B,C$ denote the coefficients of the minimal polynomial \eqref{equ-doctored-minimal-polynomial} of the parameter $z$.
\end{Kor}
\begin{proof}
This is now simply a calculation. We first collect all the ingredients we need:

\begin{enumerate}[(i)]
\item It follows from \mbox{Proposition \ref{prop-formula-for-delta}} and \mbox{Lemma \ref{lem-tame-generators}} that
\begin{equation*}
\frac{2}{3}\delta_\phi(\xi)=v_\eta(\tfrac{dt}{dz})+v_\eta(z)-v_\xi(t).
\end{equation*}
\item Differentiating the equation $0=F(z)$ shows 
\begin{equation*}
0=dF=F_zdz+F_tdt=(3z^2+2Az+B)dz+(A'z^2+B'z+C')dt,
\end{equation*}
so
\begin{equation*}
\frac{dt}{dz}=-\frac{3z^2+2Az+B}{A'z^2+B'z+C'}.
\end{equation*}
\item Since $t$ is a monomial parameter we have
\begin{equation*}
v_\xi(A't)\ge v_\xi(A),\qquad v_\xi(B't)\ge v_\xi(B),\qquad v_\xi(C't)=v_\xi(C).
\end{equation*}
To illustrate the reason that the first two inequalities are not equalities in general, suppose that $v_\xi(B)=v_\xi(b_it^i)$ for $i=0$ or $i=3$,
while $v_\xi(B)<v_\xi(b_jt^j)$ for $j\ne i$. Then 
\begin{equation*}
B't=3b_3t^3+2b_2t^2+b_1t
\end{equation*}
has strictly larger valuation than $B$. The same reasoning applies to $A$, but not to $C=c_4t^4+c_2t^2+c_1t$, whence the equality $v_\xi(C'
t)=v_\xi(C)$.
\item Since $v_\eta$ is an infinite augmented valuation over a Gauss valuation (see the proof of \mbox{Lemma \ref{lem-tame-generators}}), we have
\begin{equation*}
v_\eta\big(\sum_iA_iz^i\big)=\min_i\big\{v_\xi(A_i)+iv_\eta(z)\big\},\qquad A_i\in \BK (t).
\end{equation*}
\item Because $\overline{A\pi^{-s}}=\overline{B\pi^{-2s}}=0$, we have
\begin{equation*}
v_\eta(z)=\frac{v_\xi(C)}{3}<\frac{v_\xi(B)}{2},\qquad v_\eta(z)=\frac{v_\xi(C)}{3}<v_\xi(A).
\end{equation*}
\item Combining (iii) and (v) shows
\begin{equation*}
\frac{v_\xi(C't)}{3}=\frac{v_\xi(C)}{3}<\frac{v_\xi(B)}{2}\le\frac{v_\xi(B't)}{2},
\end{equation*}
and similarly $v_\xi(C't)/3=v_\xi(C)/3<v_\xi(A't)$.
\end{enumerate}

Putting all this together, we can compute:
\begin{IEEEeqnarray*}{Clr}
&\frac{2}{3}\delta_\phi(\xi)\\=&v_\eta(\tfrac{dt}{dz})+v_\eta(z)-v_\xi(t)&\textrm{(by (i))}\\
=&v_\eta(3z^2+2Az+B)-v_\eta(A'z^2+B'z+C')+v_\eta(z)-v_\eta(t)\qquad\qquad\qquad&\textrm{(by (ii))}\\
=&v_\eta(3z^2+2Az+B)-v_\eta(A'tz^2+B'tz+C't)+v_\eta(z)&\\
=&\min\{v_\eta(3z^2),v_\eta(2Az),v_\eta(B)\}&\textrm{(by (iv))}\\
&-\min\{v_\eta(A'tz^2),v_\eta(B'tz),v_\eta(C't)\}+v_\eta(z)\qquad&\\
=&\min\big\{v_K(3)+\frac{2v_\xi(C)}{3},v_\xi(A)+\frac{v_\xi(C)}{3},v_\xi(B)\big\}&\textrm{(by (v))}\\&-\min\big\{v_\xi(A't)+\frac{2v_\xi(C)}{3},v_\xi(B't)+\frac{v_\xi(C)}{3},v_\xi(C't)\big\}+\frac{v_\xi(C)}{3}\\
=&\min\big\{v_K(3)+\frac{2v_\xi(C)}{3},v_\xi(A)+\frac{v_\xi(C)}{3},v_\xi(B)\big\}-v_\xi(C)+\frac{v_\xi(C)}{3}&\textrm{(by (vi))}\\
=&\min\big\{v_K(3), v_\xi(A)-\frac{v_\xi(C)}{3},v_\xi(B)-\frac{2v_\xi(C)}{3}\big\}&
\end{IEEEeqnarray*}
\end{proof}

\begin{Rem}\label{rem-final-delta-formula}
We have derived \mbox{Corollary \ref{cor-concrete-delta-formula}} under the assumption that $\xi\in[x_0,\infty]$ is a Type II point with $\delta_\phi(\xi)>0$. If $\delta_\phi(\xi)=0$, then the right-hand side of \eqref{equ-concrete-delta-formula} is $\le0$. To see this, it suffices to show that
\begin{equation*}
v_\xi(A)\le\frac{v_\xi(C)}{3}\qquad\textrm{or}\qquad\frac{v_\xi(B)}{2}\le\frac{v_\xi(C)}{3}.
\end{equation*}
If neither one of these inequalities were true, then the polynomial $\overline{F}$ from \eqref{equ-reduction-polynomial} would take the form
\begin{equation*}
\overline{F}=T^3+\overline{C\pi^{-3s}}.
\end{equation*}
Since $\overline{C\pi^{-3s}}$ is by construction not a third power, this means that the extension of residue fields $\kappa(\eta)/\kappa(\xi)$ is inseparable, which is impossible unless $\delta_\phi(\xi)>0$ (\mbox{Remark \ref{rem-top-ram-closed}(c)}).

All in all we thus see that for \emph{any} Type II point $\xi\in[x_0,\infty]$ we have
\begin{equation*}
\delta_\phi(\xi)=\max\Big\{0,\min\big\{\frac{3}{2},\frac{3v_\xi(A)}{2}-\frac{v_\xi(C)}{2},\frac{3v_\xi(B)}{2}-v_\xi(C)\big\}\Big\}.
\end{equation*}
By continuity of $\delta_\phi$, this equality extends to any point $\xi\in[x_0,\infty]$ outright.
\end{Rem}

\section{Implementation}\label{sec-implementation}

\noindent We have seen in the previous section how to compute $\delta_\phi$ on a given interval in $(\BP^1_{\BK })^{\an}$. Applying results from \cite{micu}, we now explain how to use this to describe $\delta_\phi$ on $(\BP^1_K)^{\an}$. We denote by $\pi$ the natural base change map $(\BP_\BK^1)^{\an}\to(\BP_K^1)^{\an}$.

Let us consider a closed point in $\BP^1_K\setminus\{\infty\}=\Spec K[x]$, say the one corresponding to the monic irreducible polynomial $\psi\in K[x]$. Then the points in $\BP_{\BK }^1$ above this closed point correspond to the $s\coloneqq\deg(\psi)$ zeros of $\psi$ in $\BK $. We write
\begin{equation}\label{equ-numbering-of-roots}
\psi=\prod_{i=1}^s(x-\alpha_i)^e,\qquad\alpha_1,\ldots,\alpha_s\in \BK ,\quad e\ge1,
\end{equation}
and for $r\in\BR\cup\{\infty\}$ define the sets
\begin{equation*}
I_r\coloneqq\{i\in\{1,\ldots,s\}\mid v_K(\alpha_1-\alpha_i)\ge r\},
\end{equation*}
\begin{equation*}
J_r\coloneqq\{i\in\{1,\ldots,s\}\mid v_K(\alpha_1-\alpha_i)<r\}.
\end{equation*}
Now following \cite[\mbox{Definition 2.25}]{micu}, we define the function
\begin{equation*}
\theta_\psi\colon\BR\cup\{\infty\}\to\BR\cup\{\infty\},\qquad r\mapsto er\abs{I_r}+e\sum_{i\in J_r}v_K(\alpha_1-\alpha_i).
\end{equation*}
The function $\theta_\psi$ is independent of the numbering of the roots in \eqref{equ-numbering-of-roots} (\cite[\mbox{Lemma 2.27}]{micu}) and is strictly monotonically increasing (\cite[\mbox{Lemma 2.28}]{micu}).

\begin{Prop}\label{prop-discoid-splitting}
For any $r\in\BQ$ we have
\begin{equation*}
\pi^{-1}(\rD[\psi,\theta_\psi(r)])=\bigcup_{i=1}^s\rD[\alpha_i,r].
\end{equation*}
\end{Prop}
\begin{proof}
This is a special case of \cite[\mbox{Theorem 2.29}]{micu}.
\end{proof}

Now that we have assembled all the tools we need, the following algorithm for computing $\phi$-semistable models presents itself.

\begin{Algo}\label{alg-interior-discoid-locus}Input: A smooth plane quartic curve $Y$ over $K$, equipped with a degree-$3$ morphism $\phi\colon Y\to\BP^1_{K}$, given in the normal form discussed in \mbox{Section \ref{sec-quartics}}.
\begin{enumerate}[(1)]
\item For each finite branch point $P\in\BP^1_K$ of $\phi$, do the following steps to compute a subdomain $\rU_P$ associated to $P$:
\begin{itemize}
\item[(1a)] Choose a root $x_0\in \BK $ of $\psi$, where $\psi\in K[x]$ is the monic irreducible polynomial generating the maximal ideal in $K[x]$ corresponding to $P$.
\item[(1b)] Write $t\coloneqq x-x_0$. As in \mbox{Section \ref{sec-interior-locus}}, find a generator of $\BK (Y)$ whose minimal polynomial over $\BK (t)$, say
\begin{equation*}
T^3+AT^2+BT+C,\qquad A,B,C\in \BK [t],
\end{equation*}
has coefficients of the form $A=a_2t^2+a_1t+a_0$, $B=b_3t^3+b_2t^2+b_1t+b_0$, and $C=c_4t^4+c_2t^2+c_1t$. Note that by Remark \ref{rem-everything-rational} we may assume all coefficients of $A,B,C$ to lie in a finite intermediate extension $\BK /L/K$.
\item[(1c)] Using the family of Gauss valuations
\begin{equation*}
v_r\colon \BK(t)\to\BR\cup\{\infty\},\qquad\sum_id_it^i\mapsto\min_i\{v_K(d_i)+ir\},\qquad r\in\BR\cup\{\infty\},
\end{equation*}
we compute $\delta_\phi$ on the path $[x_0,\infty]$:
\begin{equation*}
\delta_\phi(\xi_r)=\max\Big\{0,\min\big\{\frac{3}{2}v_K(3),\frac{3v_r(A)}{2}-\frac{v_r(C)}{2},\frac{3v_r(B)}{2}-v_r(C)\big\}\Big\}.
\end{equation*}
\item[(1d)] Define
\begin{equation*}
\rU_P\coloneqq \ret_{\Gamma_0}^{-1}\big(\{\xi_{\psi,\theta_\psi(r)}\mid r\in[-\infty,\infty]:\delta_\phi(\xi_r)=0\}\big),
\end{equation*}
where $\Gamma_0\subset (\BP^1_K)^{\an}$ is as in \mbox{Section \ref{sec-tame-locus}}, and where $\xi_{\psi,\theta_\psi(r)}$ denotes the boundary point of $\rD[\psi,\theta_\psi(r)]$.
\end{itemize}
\item Compute the subdomain $\rU$ as explained in \mbox{Section \ref{sec-tame-locus}}.
\item Then $\rU^{\interior}=\bigcup_P\rU_P$ and $\rU^{\tail}=\rU\setminus\rU^{\interior}$.
\item Take a potentially semistable model $\CX$ of $\BP^1_K$ separating the branch locus of $\phi$ and the point $\infty$ and supporting all valuations corresponding to boundary points of $\rU^{\tail}$ and $\rU^{\interior}$.
\end{enumerate}
Then $\CX$ is potentially $\phi$-semistable.
\end{Algo}
\begin{proof}[Proof of correctness of the algorithm]
The interior locus $\rU^{\interior}$ equals the retraction
\begin{equation*}
\ret_{\Gamma_0}^{-1}\big(\{\xi\in\Gamma_0\mid\delta_\phi(\xi)=0\}\big).
\end{equation*}
The inverse image of $\Gamma_0$ in $(\BP^1_\BK)^{\an}$ is the union of the intervals $[x_0,\infty]$ considered in the algorithm. Thus the equality $\rU^{\interior}=\bigcup_P\rU_P$ follows from \mbox{Proposition \ref{prop-discoid-splitting}}; the correctness of the formula for $\delta_\phi$ in Step (1d) was established in \mbox{Remark \ref{rem-final-delta-formula}}.

Now the equality $\rU^{\tail}=\rU\setminus\rU^{\interior}$ follows from \mbox{Theorem \ref{thm-main}} and the definition of $\rU^{\tail}$. That the algorithm indeed yields a potentially $\phi$-semistable model is immediate from \mbox{Proposition \ref{prop-potentially-phi-semistable}}. 
\end{proof}

\mbox{Algorithm \ref{alg-interior-discoid-locus}} is implemented in a branch of the MCLF SageMath package (\cite{mclf}), available on Github:
\begin{center}
\href{https://github.com/oossen/mclf/tree/plane-quartics}{https://github.com/oossen/mclf/tree/plane-quartics}
\end{center}
Many concrete examples are discussed in \cite[\mbox{Section 5.5}]{diss}, including the SageMath prompts used. We discuss one example below.

Once we have found a $\phi$-semistable model $\CX$ of $\BP_K^1$, we can use the existing functionality of the MCLF package to compute the semistable reduction of $Y$. In short, MCLF determines a field extension $L/K$ over which $Y$ has semistable reduction and then computes the extensions of the valuations corresponding to $\CX$ to the function field $L(Y)$. We refer to \cite[\mbox{Section 5.4}]{diss} for details.

\begin{Ex}\label{ex-worked-example-1}Consider the plane quartic over $K=\BQ_3$
\begin{equation*}
Y\colon\quad F(y)=y^3 + (2x^3 + 3x^2)y - 3x^4 - 2x^2 - 1=0.
\end{equation*}
Applying \mbox{Algorithm \ref{alg-interior-discoid-locus}} to this example reveals that the tame locus associated to $Y$ is a union of three discoids,
\begin{equation*}
\rU^{\tame}=\rD[x,3/4]\cup\rD[1/x,0]\cup\rD[\Delta_F,6].
\end{equation*}
(The SageMath commands used to obtain this may be found in \cite[\mbox{Example 5.15}]{diss}.) The tail locus is the discoid $\rD[x,3/4]$; the interior locus consists of a disk centered at the branch point $\infty$ and of a discoid containing the only finite branch point, which is cut out by the discriminant $\Delta_F$, an irreducible polynomial of degree $9$.

The tree spanned by the boundary points of the three discoids making up the tame locus may be visualized as in \mbox{Figure \ref{fig-1}}. The locus where $\delta_\phi>0$ is in red.

\begin{figure}[H]\centering\includegraphics[scale=0.4]{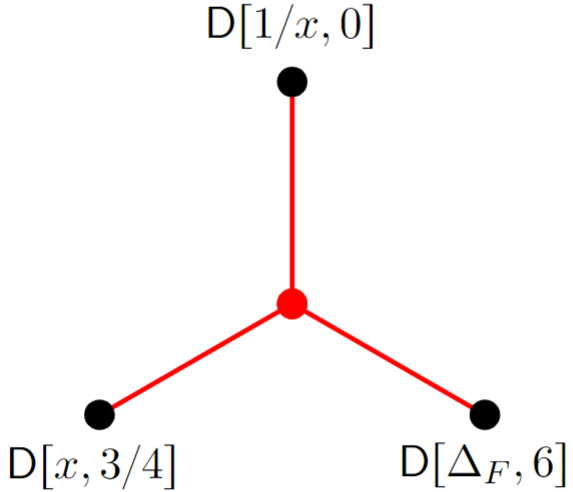}\caption{The tree in in $(\BP^1_K)^{\an}$ spanned by the boundary points of the tame locus associated to the tame locus of the curve considered in \mbox{Example \ref{ex-worked-example-1}}. This figure is adapted from a figure in \cite[\mbox{Example 5.15}]{diss}}\label{fig-1}
\end{figure}

It follows easily from \mbox{Proposition \ref{prop-discoid-splitting}} that $\rD[\Delta_F,6]$ splits into $9$ disjoint disks over $\BK$. Thus the boundary points of these disks play no role in finding a semistable reduction of $Y$. Using the functionality of the MCLF SageMath package (\cite{mclf}) for computing normalized base changes, one finds that the model corresponding to the boundary points of the two disks $\rD[x,0]$ and $\rD[x,3/4]$ indeed is potentially $\phi$-semistable. The reduction curve of the point in $Y_\BK^{\an}$ above the boundary point of $\rD[x,3/4]$ is the Artin--Schreier curve of genus $1$ with equation
\begin{equation*}
y^3+y+x^2=0.
\end{equation*}
The reduction curve of the point in $Y^{\an}_\BK$ above the boundary point of $\rD[x,0]$ is of genus $2$ and is given by 
\begin{equation*}
y^3-x^3y+x^2-1=0.
\end{equation*}
We note that it has ordinary ramification over $x=\infty$ (the specialization of the ordinary branch point $\infty$ of $\phi$) and that it is totally ramified over $x=0$ (reflecting the fact that $\phi$ is wildly ramified on the branch pointing towards $x=0$).

In summary, we have found one component of genus $1$ and one component of genus $2$ on the special fiber of the model determined by the boundary points of the tame locus. To obtain the stable reduction, we can contract all other components. This corresponds to discarding the central red vertex and the boundary point of $\rD[\Delta_F,6]$ in the tree above, keeping only the boundary points of $\rD[1/x]$ and $\rD[x,3/4]$.

The resulting reduction type is given by a curve of genus $1$ and a curve of genus $2$, meeting transversely in a single point. We note that it is not difficult to compute the thickness of this node; see \cite[End of \mbox{Section 2}]{ossen} for a worked example.
\end{Ex}

\section{An example over $\BQ_3(\zeta_3)$}\label{sec-modular-curve-example}

\noindent In this section we apply our results to a slightly more involved example. Namely, we compute the semistable reduction of a particular plane quartic curve $Y$ that appears in the attempts of Rouse, Sutherland, and Zureick-Brown to compute the rational points on the non-split Cartan modular curve $X_{\ns}^+(27)$, see \cite{rouse-sutherland-zureickbrown}. In upcoming work, Balakrishnan, Jha, Hast, and Müller use the reduction to compute the rational points of $Y$.

Let $K=\BQ_3(\zeta_3)$ be the field obtained by adjoining to $\BQ_3$ a primitive third root of unity. We normalize the valuation $v_K\colon \BK \to\BQ\cup\{\infty\}$ so that $v_K(3)=1$, that is to say $v_K(\zeta_3-1)=1/2$. The curve we are interested in is the plane quartic $K$-curve cut out of $\BP_{K}^2$ by the equation
\begin{equation*}
\label{eq-quartic-as-given}
\begin{aligned}
x^4 &+ (\zeta_3 - 1)x^3y + (3\zeta_3 + 2)x^3z - 3x^2z^2 + (2\zeta_3 + 2)xy^3 - 3\zeta_3 xy^2z&\\
&+ 3\zeta_3 xyz^2 - 2\zeta_3 xz^3 - \zeta_3 y^3z + 3\zeta_3 y^2z^2 + (-\zeta_3 + 1)yz^3 + (\zeta_3 + 1)z^4=0.
\end{aligned}
\end{equation*}
Note that $[0:1:0]$ is a rational point on $Y$. We achieve the normal form discussed in \mbox{Section \ref{sec-quartics}} by replacing $z$ with $z+x(2\zeta_3+2)/\zeta_3$. Dehomogenizing yields the affine equation
\begin{equation*}
Y\colon\quad F(y)=y^3+Ay^2+By+C=0,
\end{equation*}
where
\begin{equation*}
A=(6\zeta_3 + 12)x^2+ (36\zeta_3 + 9)x- 27,
\end{equation*}
\begin{equation*}
B=(9\zeta_3 - 18)x^3+ (-108\zeta_3 - 108)x^2+ (-162\zeta_3 + 81)x+ (81\zeta_3 + 162),
\end{equation*}
\begin{equation*}
C=(27\zeta_3 - 243)x^4+ (-1458\zeta_3 - 999)x^3+(-1215\zeta_3 + 1701)x^2+ (1944\zeta_3 + 2430)x+ 729\zeta_3.
\end{equation*}
We use the usual projection map $\phi\colon Y\to\BP^1_{K}$. The discriminant $\Delta_F$ has degree $10$, so the point $\infty$ is not a branch point of $\phi$ (cf.\ \cite[\mbox{Lemma 4.1}]{diss}). The branch locus of $\phi$ consists of one point of degree $1$ and one point of \mbox{degree $9$}.

We note that SageMath code for defining $Y$ and computing its tame locus may be found in \cite[\mbox{Code Listing A.6}]{diss}. The tame locus consists of three discoids, configured as in \mbox{Figure \ref{fig-2}}.

\begin{figure}[!htb]
\centering
\begin{minipage}{.5\textwidth}
\centering
\includegraphics[scale=0.4]{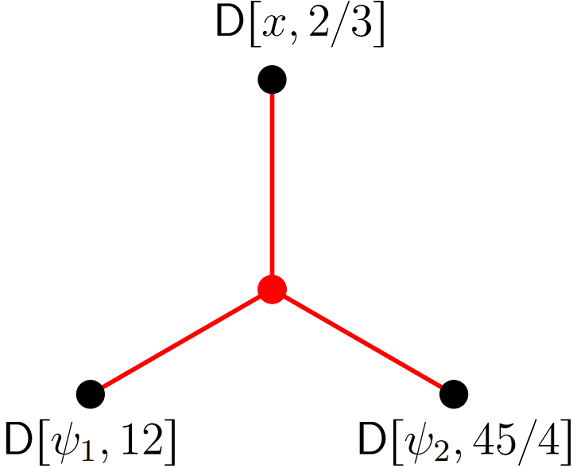}
\caption{The tree in $(\BP_K^1)^{\an}$ spanned by the boundary points of the tame locus associated to $Y$. This figure is from \cite[\mbox{Section 5.6}]{diss}}
\label{fig-2}
\end{minipage}%
\begin{minipage}{.5\textwidth}
\centering
\includegraphics[scale=0.26]{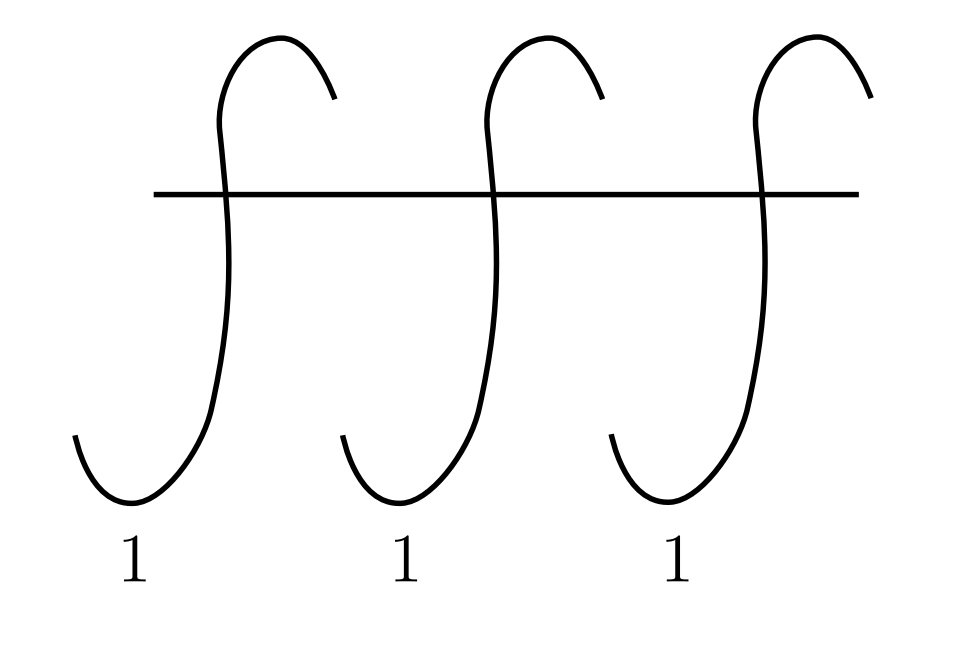}
\captionof{figure}{The stable reduction of the curve $Y$. This figure is from \cite[\mbox{Example 5.17}]{diss}}
\label{fig-3}
\end{minipage}
\end{figure}

The disk $\rD[x,2/3]$ contains $\infty$ and the branch point of degree $1$. The polynomial
\begin{equation*}
\begin{aligned}
\psi_1&=x^9 + (9\zeta_3 + 18)x^8 + (54\zeta_3 - 27/2)x^7 + (54\zeta_3 + 108)x^6 + 486\zeta_3x^5+ (-729\zeta_3 + 3645)x^4 \\&+ (729\zeta_3 + 13851)x^3+ (37179\zeta_3 + 15309)x^2 + (-6561\zeta_3/2 + 72171)x + 2187\zeta_3/13 + 15309/2
\end{aligned}
\end{equation*}
is the factor of $\Delta_F$ of degree $9$, so $\rD[\psi_1,12]$ contains the other branch point of $\phi$. The last component, the tail discoid $\rD[\psi_2,45/4]$, is the most important; we have
\begin{equation*}
\begin{aligned}
\psi_2&=x^9 + (9\zeta_3 - 9)x^8 + (54\zeta_3 + 27)x^7 + (54\zeta_3 - 27/2)x^6+ (243\zeta_3 + 972)x^5 + 729\zeta_3 x^4\\ &+ (2916\zeta_3 - 1458)x^3 + (37179\zeta_3 + 41553)x^2 + (6561\zeta_3 + 6561/8)x - 63423\zeta_3 + 155277.
\end{aligned}
\end{equation*}
We also note that the red vertex in the middle of the tree in Figure \ref{fig-2} is the boundary point of the discoid $\rD[\psi_1,11]=\rD[\psi_2,11]$. The splitting behavior of all these discoids is particularly interesting. While the discoids $\rD[\psi_2,45/4]$ and $\rD[\psi_2,11]$ split into three disks over $\BK$ each containing three of the nine roots of $\psi_2$, the discoid $\rD[\psi_1,12]$ splits into nine disks, each containing one root of $\psi_1$. 

Since there are three tail disks, the structure of the stable reduction of $Y$ is clear: It must be a ``comb'' consisting of one rational component intersecting three components of genus $1$, which are Artin--Schreier curves given by
\begin{equation*}
y^3-y=x^2,
\end{equation*}
configured as in \mbox{Figure \ref{fig-3}}.

The tree spanned by all the roots of $\Delta_F$ and $\psi_2$ is depicted in \mbox{Figure \ref{fig-big-modular-tree}}. As usual, the locus where $\delta_\phi>0$ is depicted in red. The dashed gray lines represent disks of the indicated radius. For example, the smallest disk containing all roots of both $\psi_1$ and $\psi_2$ has radius $7/6$. We see that the nine disks into which $\rD[\psi_1,12]$ splits have radius $2$, while the three disks into which $\rD[\psi_2,45/4]$ splits have radius $17/12$. Given a root of either $\psi_1$ or $\psi_2$, the closest other root has distance $3/2$.

\begin{Rem}
\begin{enumerate}[(a)]
\item In an earlier version of this manuscript (\cite{ossen}), a field extension of $K$ over which $Y$ has semistable reduction is determined. It is of ramification index $54$ over $K$.
\item In \cite[\mbox{Section 3}]{ossen}, a coordinate transformation is used to ensure the curve $Y$ has an inflection point at infinity. It turns out that in these new coordinates, the strategy of separating branch points works to compute the semistable reduction of $Y$! We have no answer to the interesting question of whether such a coordinate transformation may be found in other examples as well; in this example, this transformation was found by accident.
\end{enumerate}
\end{Rem}

\begin{figure}[htb]\centering\includegraphics[scale=0.55]{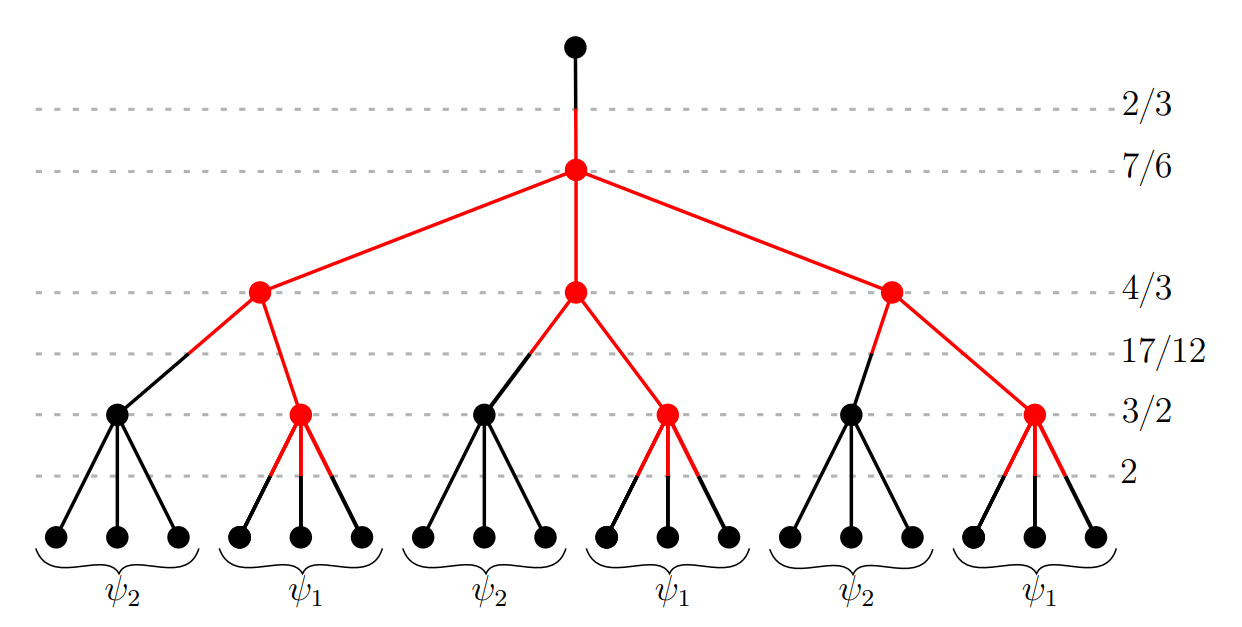}
\caption{The tree in $(\BP_\BK^1)^{\an}$ spanned by the zeros of $\Delta_F$ and $\psi_2$. This figure is from \cite[\mbox{Figure 5.3}]{diss}}\label{fig-big-modular-tree}
\end{figure}  

\FloatBarrier

\subsection*{Data availability} The code used for computations in the examples above may be found in \cite{diss}, to which appropriate references are given. This code is based on a branch of the MCLF SageMath package, available on Github under \href{https://github.com/oossen/mclf/tree/plane-quartics}{https://github.com/oossen/mclf/tree/plane-quartics}.

\subsection*{Competing interests statement}

This research was partially carried out while the author was employed as a research assistant at the Institute of Algebra and Number Theory at Ulm University. There are no competing interests with relevance to the contents of this article to declare.

\bibliographystyle{plain}
\bibliography{new_sources}

\end{document}